\documentclass{article}[12pt]

\usepackage{amsmath, amsthm, amsfonts, mathrsfs, amsfonts}
\usepackage{xcolor}
\usepackage{bm}
\usepackage[english]{babel}
\usepackage{amscd}
\usepackage{amssymb}
\usepackage{enumerate} 
\usepackage[all,cmtip]{xy}
\usepackage{ytableau}
\usepackage{tikz}
\usetikzlibrary{matrix,arrows}
\usepackage{enumitem}
\usepackage{stmaryrd}
\usepackage[normalem]{ulem}
\usepackage[new]{old-arrows}
\usepackage{titlesec}

\usepackage[a4paper]{geometry}
\usepackage[final,bookmarksnumbered=true]{hyperref}
\hypersetup{
	colorlinks = true,
	linkcolor = blue,
	anchorcolor = blue,
	citecolor = blue,
	filecolor = blue,
	urlcolor = blue
}

\newcommand{\F}{\mathbb{F}}

\newcommand{\Z}{\ensuremath{\mathbb{Z}}}

\newcommand{\cor}[1]{\mathcal{#1}}

\newcommand{\graffe}[1]{\{#1\}}

\newcommand{\gen}[1]{\langle #1\rangle}

\DeclareMathOperator{\Aut}{Aut} 
 
\DeclareMathOperator{\im}{Im} 
\DeclareMathOperator{\hc}{H} 
\DeclareMathOperator{\ZG}{Z}

\DeclareMathOperator{\Hom}{Hom}

\DeclareMathOperator{\id}{id}

\DeclareMathOperator{\Sym}{Sym} 
\DeclareMathOperator{\dG}{d} 
\DeclareMathOperator{\lL}{\ell L} 
\renewcommand{\L}{\mathrm{L}}

\DeclareMathOperator{\ext}{Ext}
\DeclareMathOperator{\hab}{H^2_{ab}}

\DeclareMathOperator{\Ext}{Ext}

\newtheorem{definition}{Definition}[section]
\newtheorem{lemma}[definition]{Lemma}
\newtheorem{theorem}[definition]{Theorem}
\newtheorem{proposition}[definition]{Proposition}
\newtheorem{corollary}[definition]{Corollary}

\newtheorem*{acknowledgements}{Acknowledgements}

\newtheorem{theoremA}{Theorem}

\newenvironment{ex}[1][]{\refstepcounter{definition}\par\medskip\noindent \textbf{Example~\thedefinition. #1} \rmfamily}{\medskip}

\newenvironment{remark}[1][]{\refstepcounter{definition}\par\medskip
	\noindent \textbf{Remark~\thedefinition. #1} \rmfamily}{\medskip}

\title{Orbits classifying extensions of prime power order groups}

\date{}

\author{Oihana Garaialde Oca\~na and Mima Stanojkovski \thanks{ 
		\noindent The first author was  
supported by the Spanish Government project PID2020-117281GB-I00,  
partially by FEDER funds and, by the Basque Government project  
IT483-22.}}

\begin{document}

	\maketitle

	\begin{abstract}
		The strong isomorphism  classes of extensions of finite groups are parametrized by orbits of a prescribed action on the second cohomology group. We study these orbits in the case of extensions of a finite abelian $p$-group by a cyclic factor of order $p$. As an application, we compute the number and sizes of these orbits when the initial $p$-group is generated by at most $3$ elements. 
	\end{abstract}
	
	\thispagestyle{empty}

	\section{Introduction}

	An established way of constructing finite groups is via \emph{group extensions}. A group $E$ is said to be an \emph{extension} of a group $G$ by a group $N$ if there exists a short exact sequence of groups
	\begin{equation}\label{eq:SES}
		1\rightarrow N {\longrightarrow} E {\longrightarrow} G \rightarrow 1.
	\end{equation}
	Every finite group can be constructed inductively in this way by iterating extensions by simple (composition) factors. In particular, if $p$ is a prime number, then every finite $p$-group can be  realized 
	via consecutive extensions with kernel $N$ of order $p$ and, moreover, such extensions are \emph{central} (it is indeed well-known that non-trivial $p$-groups have non-trivial center). An extension like \eqref{eq:SES} is called central if $N$ is central in $E$ equivalently, if the action of $G$ on $N$ is trivial.
	Every group of order $p^n$ being a central extension of a group of order $p^{n-1}$ by $\F_p$, one could hope to \emph{classify $p$-groups by classifying extensions}.
	The famous $p$-group generation algorithm of Newman and O'Brien \cite{OBrien/90} builds upon a  structural refinement of this idea. 
	
	A challenging task in the framework of classifying groups via extensions is that of determining whether two extensions $E$ and $E'$ are isomorphic as groups, in symbols $E\cong E'$. Because of this, it is sometimes worth it to start by testing isomorphism in a slightly stronger form.
	Two group extensions 
	\[
	1\rightarrow N \overset{\iota}{\longrightarrow} E {\longrightarrow} G \rightarrow 1 \ \ \textup{ and } \ \ 1\rightarrow N \overset{\iota'}{\longrightarrow} E' {\longrightarrow} G \rightarrow 1
	\]
	of $G$ by $N$ are \emph{strongly isomorphic} (following \cite[Def.\ 17.20]{Fitting/38}), denoted $E\cong_s E'$, if there exists an isomorphism $\phi:E\rightarrow E'$ 
	inducing an isomorphism $\iota(N)\rightarrow\iota'(N)$. The extensions $E$ and $E'$ are \emph{equivalent}, denoted $E\sim E'$, if $\phi$ induces the identity on both $\iota(N)\rightarrow\iota'(N)$ and $G\rightarrow G$. In particular, it holds that
	\[
	E\sim E'\ \Longrightarrow\ E\cong_s E'\ \Longrightarrow\ E\cong E'
	\]
	which in a straightforward manner implies that
	\begin{equation}\label{eq:classes}\nonumber
		\#\graffe{\text{isomorphism classes}}\leq 
		\#\graffe{\text{strong isomorphism classes}}\leq 
		\#\graffe{\text{equivalence classes}}.
	\end{equation}
	The equivalence classes of extensions of $G$ by $N$ are in bijection with the elements of the \emph{second cohomology group} $\hc^2(G;N)$, while the strong isomorphism classes are parametrized by orbits of $A=\Aut(G)\times \Aut(N)$ on $\hc^2(G;N)$; cf.\ Theorem \ref{th orbits and strong exts}. If $C^2(G;N)$ denotes the collection of $2$-cocycles $G\times G\rightarrow N$ and composition in $\Aut(G)$ is taken from right to left (i.e.\ $\tau\circ\sigma(x)=\tau(\sigma(x))$), then the action of $A$ on $C^2(G;N)$ is defined from the following data:
	\begin{itemize}
		\item the right diagonal action of $\Aut(G)$ on $C^2(G;N)$ given by
		\[
		C^2(G;N)\times\Aut(G)\longrightarrow C^2(G;N), \quad (c,\sigma) \longmapsto ((x,y)\mapsto c(\sigma(x), \sigma(y))),
		\]
		\item the natural left action of $\Aut(N)$ on $C^2(G;N)$ given by 
		\[
		\Aut(N)\times C^2(G;N)\longrightarrow C^2(G;N), \quad (\lambda,c) \longmapsto ((x,y)\mapsto \lambda( c(x,y))).
		\]
	\end{itemize}
	The last actions respect coboundaries and therefore, if $\Aut(N)$ is abelian, we derive the following left action of $A$ on $\hc^2(G;N)$:
	\begin{equation}\nonumber
		A\longrightarrow\Sym(\hc^2(G;N)), \quad
		(\sigma,\lambda) \ \mapsto \ ([c]\mapsto [\lambda c \sigma^{-1}]),
	\end{equation}
	where $[c]$ denotes the cohomology class of $c$.
	The following result is a weaker version of \cite[Thm.\ $4.7$]{BescheEick}.  
	
	\begin{theoremA}
		\label{th orbits and strong exts}
		Let $p$ be a prime number, $G$ a finite group, and $N$ a trivial $\F_pG$-module. Then the set of strong isomorphism classes of extensions of $G$ by $N$ is in natural bijection with the collection of orbits of the action of $A$  on $\hc^2(G;N)$.
	\end{theoremA}
	
	\noindent
	We have decided to state the last result only in terms of central extensions, because those are the ones we will be concerned with. 
	The more general version from \cite{BescheEick} allows $N$ to be any $\F_p G$-module (actually the proof works for any $\Z G$-module) and parametrizes strong isomorphism classes in terms of an action of the \emph{compatible pairs} of $A$ (in our case, all elements of $A$). Compatible pairs were introduced in \cite{DJSRobinson81} in the context of computing automorphism groups of extensions. A version of Theorem \ref{th orbits and strong exts} for non-fixed module structure on $N$ can be found in \cite[Satz 1.2]{Laue/82}. 
	Many are the applications of Theorem \ref{th orbits and strong exts} in the literature: see for example \cite{BescheEick},\cite{EickOBrien/99},\cite{DietrichEick/05},\cite{DietrichEickFeichtenschlager/08},\cite{GrochowQiao/17}. 
	Moreover, results similar to Theorem \ref{th orbits and strong exts} are employed to count Lie algebras by extensions; see for instance \cite[Thm.\ 2]{LieAlgebras}. 
	
	Despite their relevance to the isomorphism problem for finite groups, not much is known about the sizes of the orbits from Theorem \ref{th orbits and strong exts}. In the present paper, we concern ourselves with the case in which $G$ is an abelian $p$-group and $N=\F_p$: our goal is to determine the orbits of the action of $\tilde{A}=\Aut(G)\times \F_p^*$ on $\hc^2(G;\F_p)$.  We remark that, under these last assumptions, the extensions parametrized by $\hc^2(G;\F_p)$ are abelian or with commutator subgroup of order $p$. The latter class of groups has been classified in
	\cite{SBlackburn} with respect to the group order and relies on the classification of bilinear forms. Our techniques are different and 
	the results are difficult to compare outside of small order cases. 
	Moreover, we hope that our approach can be generalized to the study of extensions where $N$ is cyclic or elementary abelian.
	
	\subsection{Summary of the main results}
	
	Let $p$ be an odd prime number and let $G$ be a finite abelian $p$-group. In this paper we are concerned with the orbits of the action of $\tilde{A}=\Aut(G)\times\F_p^*$ on $\hc^2(G;\F_p)$, where $\F_p$ is viewed as a trivial $\F_pG$-module. In this very case, such orbits parametrize the isomorphism classes of extensions of $G$ by $\F_p$, see Proposition~\ref{prop:SIC=IC}, and we determine them completely when $G$ is generated by at most $3$ elements. For a minimal generating set of larger size, we describe the orbits within a specific $\tilde{A}$-stable subset of $\hc^2(G;\F_p)$ as we now explain. 
	
	Under our assumptions, $\hc^2(G;\F_p)$ is an $\F_p$-vector space endowed with a map $$\cup: \Hom(G,\F_p)\times\Hom(G,\F_p)\rightarrow\hc^2(G;\F_p)$$ corresponding to the restriction of the \emph{cup product} in the full cohomology ring of $G$. A distinguished subspace of $\hc^2(G;\F_p)$ is $\Ext^1_{\Z_pG}(G,\F_p)$, which parametrizes the equivalence classes of abelian extensions of $G$ by $\F_p$ and, together with the $\F_p$-span of the image of $\cup$, figures in the following convenient decomposition as $\F_p\tilde{A}$-modules:
	$\hc^2(G;\F_p)=\Ext^1_{\Z_pG}(G,\F_p)\oplus \langle \im\cup \rangle$. 
	
	The $\tilde{A}$-stable subset we analyze is $\Ext_{\Z_pG}^1(G,\F_p)\times\im\cup$ and we do this ``projectively''.
	We write $V=G/pG$, $d=\dim_{\F_p}(V)$, and $\cor{G}(k,V)$ for the collection of subspaces of dimension $k$ of $V$. We show that there is a somewhat natural bijection of $\tilde{A}$-sets 
	$$
	\mathbb{P}\Ext^1_{\Z_pG}(G,\F_p)\times\mathbb{P}\im\cup\rightarrow \cor{G}(d-1,V)\times\cor{G}(d-2,V)
	$$ 
	which shifts the original problem to the determination of $\Aut(G)$-orbits of pairs of subgroups of $G$. Our main Theorem \ref{th:main} gives a combinatorial description of the $\tilde{A}$-orbits of $\Ext_{\Z_pG}^1(G,\F_p)\times\im\cup$ in terms of vectors of data parametrizing the $\tilde{A}$-orbits of such pairs and thus allows the computation of the orbit sizes.
	Moreover, this result yields a lower bound on the number of isomorphism types of extensions of $G$ by $\F_p$ and, specifically, 
	the number of isomorphism classes of extensions with centre of index at most $p^2$. 
	It is worth mentioning that the orbit sizes are, under our assumptions, given by vectors of polynomials in $p$. Though maybe not quite surprising given the ``low complexity'' of the groups we consider, this raises the question of whether this is always the case.
	
	We remark that our results also hold true for many $2$-groups; see Section \ref{subsec:assumptions}.

	\subsection{Assumptions and notation}\label{subsec:assumptions}
	In this section, we set the notation that will hold throughout the whole paper. 
	Let $p$ be a prime number and let $G$ be a finite abelian $p$-group, written in additive notation, of exponent $\exp(G)=p^n$ and with $\dG(G)=r+1\geq 1$, i.e.\ $G$ is $(r+1)$-generated but not $r$-generated. In particular, $G$ is non-trivial and $n\geq 1$.  Let, moreover, $C$ denote a cyclic group of order $p^{n+1}$ equipped with a trivial $G$-action.  For each subgroup $K$ of $G$ and nonnegative integer $m$, we write $K[m]$ for the $m$-th torsion subgroup of $K$, i.e.\ $K[m]=\graffe{x\in K \mid mx=0}$. 
	We now fix a decomposition of $G$ into cyclic summands. For this, we let 
	\begin{itemize}
		\item $t$ a positive integer,
		\item integers $1\leq n_1\leq n_2\leq \dots \leq n_t=n$,
		\item integers $1\leq r_1, \dots, r_t$ such that $r+1=r_1+\dots +r_t$,
		\item for each $j\in \{1, \dots,t\}$ and $k\in\{1, \dots,r_j\}$, a cyclic group $I_{jk}$ of order $p^{n_j}$, 
		\item for each $j\in\graffe{1,\ldots,t}$, a free $\Z/(p^{n_j})$-module $I_j$ of rank $r_j$,
	\end{itemize}
	be such that
	\begin{equation}\nonumber
		G= \bigoplus_{j=1}^{t} I_j = \bigoplus_{j=1}^{t} \bigoplus_{k=1}^{r_j} I_{jk} \textup{ with } I_j=\bigoplus_{k=1}^{r_j} I_{jk}.
	\end{equation}
	We additionally assume that, if $p=2$, then $n_1>1$ holds in the above decomposition, equivalently $G$ does not admit cyclic factors of order $2$: the reason for this choice is clarified in Remark \ref{rmk:p=2case}. 
	
	We fix generators $\gamma_{jk}$ of $I_{jk}$ and $\tilde{\gamma}$ of $C$. We denote by $\gamma$ the image of $\tilde{\gamma}$ under the natural projection $C\rightarrow C/p^nC$, and so $\gamma$ generates $C/p^nC$.  
	Set, moreover, $V=G/pG$ and denote by $\pi$ the natural projection $G\rightarrow V$. For each $j\in \{1, \dots,t\}$ and $k\in\{1, \dots,r_j\}$, we write $v_{jk}=\pi(\gamma_{jk})$ and observe that, as a consequence of their definition, the $v_{jk}$'s form a basis of $V$. Denote by $\widehat{V}=\Hom(V,\F_p)$ the dual of $V$ of which a basis is given by the homomorphisms $v_{jk}^*: V\to \F_p$ satisfying
	\[
	v_{jk}^*(v_{hl})=\delta_{(j,k),(h,l)}=\begin{cases}
		1 & \textup{if } (j,k)=(h,l),\\
		0 & \textup{otherwise.}
	\end{cases}
	\]
	Let $\phi_1:V\rightarrow \widehat{V}$ denote the isomorphism of vector spaces defined by $v_{jk} \mapsto v_{jk}^*$. Write $\Aut(G)$ for the automorphism group of $G$ and, for each $\sigma\in\Aut(G)$, denote by $\overline{\sigma}$ the element of $\Aut(V)$ that is induced by $\sigma$. 
	Write $\Z_p$ for the ring of $p$-adic integers, $\Z_p^*$ for its group of units, and set $A=\Aut(G)\times \Z_p^*$. Denote by $\F_p$ the field of $p$ elements, considered as a trivial $\Z_pG$-module, and by $\F_p^*$ its group of units. We define a series of left actions of $A$ on sets associated to $G$.
	For $K$ a finite $\Z_pG$-module, the group $A$ acts on
	\begin{itemize}
		\item $\hc^m(G;K)$ via the map
		\begin{equation}\label{eq:ActionAonHomToK}
			A\longrightarrow\Sym(\hc^m(G;K))\textup{ defined by } (\sigma,\lambda) \ \mapsto \ ([c]\mapsto \lambda [c]\sigma^{-1}=[\lambda c \sigma^{-1}]);
		\end{equation}
		\item $V$ via the map
		\begin{equation}\label{eq:ActionAonV}
			A\longrightarrow\Aut(V)\textup{ defined by } (\sigma,\lambda) \ \mapsto \ (v\mapsto \lambda \overline{\sigma}(v));
		\end{equation}
		\item the collection $\cor{S}_G$ of subgroups of $G$ via
		\begin{equation}\label{eq:ActionAonSbgs}
			A\rightarrow\Sym(\cor{S}_G) \textup{ defined by } (\sigma,\lambda)\mapsto (H\mapsto \sigma(H));
		\end{equation} 
		\item $\mathbb{P}\hc^2(G;K)$ via the map
		\begin{equation}\label{eq:ActionAonPH^2}
			A\longrightarrow\Sym(\mathbb{P}\hc^2(G;K))\textup{ defined by } (\sigma,\lambda) \ \mapsto \ ([c]\mapsto  [c\sigma^{-1}]).
		\end{equation}
	\end{itemize}

	If two objects $X$ and $Y$ 
	belong to the same $A$-orbit, we write $X\sim_A Y$. We write $A_X$ meaning the stabilizer of $X$ in $A$ and, if two elements $Y,Z$ are in the same orbit under the induced action by $A_X$, we write $Y\sim_{A_X}Z$. To lighten the notation, if $X=[c]\in\hc^2(G;\F_p)$, we write $A_c$ instead of $A_{[c]}$.

	\subsection{Organization and strategy}
	
	We describe here briefly the internal structure of this article and the strategy behind the proofs of our main results.
	
	In Section \ref{sec:HAlgebra}, we briefly describe the cohomological objects we will be dealing with and list a number of their properties; we also provide more detailed references for the interested reader. We show in Section \ref{subsec:bockstein} that the abelian extensions of $G$ are parametrized by the elements in the image of the \emph{higher order Bockstein homomorphism}. In Sections \ref{subsec:maximalsubgroups} and \ref{subsec:plucker}, we give two correspondences involving respectively $\Ext^1_{\Z_pG}(G,\F_p)$ and the image of the cup product $\cup:\Hom(G,\F_p)\times\Hom(G,\F_p)\rightarrow \hc^2(G;\F_p)$ and interpret the $A$-orbits thereof in terms of orbits of subspaces of $V$.
	
	In Section \ref{sec:generalresults}, we define the numbers that will allow us to describe the $A$-orbits in $\hc^2(G;\F_p)$ combinatorially and prove some compatibility results regarding the correspondences defined in the previous section. Such numbers are called the \emph{levels} of the pairs of subgroups associated to a given element of $\Ext_{\Z_pG}^1(G,\F_p)\times\im\cup$ and tell us how the two subgroups ``relatively sit in $G$''. 
	
	Section \ref{section AbelianExtensions} is devoted to the analysis of the action of $A$ on $\Ext^1_{\Z_pG}(G,\F_p)$. Here we heavily rely on the properties of the Bockstein homomorphism and the equivalence relation it induces on $\Hom(G,C/(p^nC))$, which we name the \emph{Bockquivalence relation}. Roughly speaking, the Bockstein homomorphism   controls the map $x\mapsto x^p$ on the extensions of $G$ by $\F_p$. In this section, we also show that in fact the strong isomorphism classes coincide with the isomorphism classes of extensions of $G$ by $\F_p$.  
	
	In Section \ref{sec:stabilizers}, we describe the orbits of $\im\cup$ under the action of $A_c$, where $[c]$ denotes an element in $\Ext^1_{\Z_pG}(G,\F_p)$. We do this by separating the cases according to the value of the \emph{$c$-index}, which we defined in Section \ref{sec:generalresults}.
	
	Section \ref{sec:apps} collects our main result, applications of it, and some closing remarks. In Section \ref{subsec:main}, we give and prove our Main Theorem \ref{th:main} combining the efforts from Sections \ref{section AbelianExtensions} and \ref{sec:stabilizers}. In Sections \ref{subsec:2gen} and \ref{subsec:3gen} we give respectively the orbit counts for $2$-generated and $3$-generated groups. In Section \ref{subsec:moregens}, we give an example and ideas for future work.

	\begin{acknowledgements}
		The authors are very thankful to Bettina Eick for helpful feedback and clarifications regarding this project. The authors also wish to thank the universities of Bielefeld and D\"usseldorf together with the Max-Planck-Institute for Mathematics in the Sciences (in particular the groups of Christopher Voll, Benjamin Klopsch, and Bernd Sturmfels), where part of this collaboration took place. We also wish to thank Jon Gonz\' alez-S\' anchez for useful suggestions regarding this manuscript. We thank the anonymous referee for their comments, which led to an improvement in the exposition of this paper. 
	\end{acknowledgements}

	\section{Homological algebra}\label{sec:HAlgebra}
	
	The aim of this section is to set the notation that will be used in the next sections and to shortly describe the objects we will be working with. For reasons of brevity, we work under the assumptions listed in Section \ref{subsec:assumptions};
	for a more general view on the topic, we refer the reader to \cite{KSBrown}, \cite{LEvens}, \cite{CAWeibel}.

	\subsection{Cohomology of groups}\label{subsec:cohomologyofgroups}

	Throughout we suppose that, for $n\geq 0$, the $n$-th cohomology group $\hc^n(G;\F_p)$ of $G$ with coefficients in $\F_p$ is computed by applying the left-exact functor $\Hom_{\F_pG}(\cdot, \F_p)$ to the standard or bar resolution $B_n(G;\Z)$ of $\Z$, i.e.\ $\hc^n(G;\F_p)=\hc^n(C^m(G;\F_p), \partial_m)$, where
	\[
	C^m(G;\F_p)=\{f: G^m=\underbrace{G\times \dots \times G}_{m \textup{ times}}\to \F_p \; \text{functions}\}
	\]
	and
	$
	\partial_m: C^m(G;\F_p)\to C^{m+1}(G;\F_p)
	$
	is defined by sending $f \in C^m(G;\F_p)$ to
	\begin{align*}
		\partial_m(f)(g_1, \dots,g_{m+1})=&f(g_2, \dots, g_{m+1})+ \sum_{i=1}^{m}(-1)^i f(g_1, \dots, g_i+g_{i+1}, \dots, g_{m+1})\\
		&+(-1)^{m+1}f(g_1, \dots, g_m).
	\end{align*}
	If $f\in C^m(G;\F_p)$, we say that $f$ has degree $m$,  written $|f|=m$, and we denote by $[f]$ its cohomology class in $H^m(G;\F_p)$. The cohomology group of $G$ with coefficients in $\F_p$ is the graded abelian group $$\hc^*(G;\F_p)=\bigoplus_{n\geq 0} \hc^n(G;\F_p).$$ For all integers $n,m\geq 0$, the cup product $\cup: C^n(G;\F_p)\times C^{m}(G;\F_p)\to C^{n+m}(G;\F_p)$ is defined by sending the pair
	$(c,d)\in C^{n}(G;\F_p)\times C^m(G;\F_p)$ to the map $G^n\times G^m\rightarrow\F_p$ that is given by
	$$(x,y)\longmapsto (c\cup d)(x,y)=c(x)d(y).$$  
	By slight abuse of notation, we also write $\cup$ for the induced cup product in cohomology $$\cup:\hc^n(G;\F_p)\times \hc^m(G;\F_p)\longrightarrow \hc^{n+m}(G;\F_p),$$ i.e., for each $x\in G^n$ and $y\in G^m$, the cup product of $[c]\in \hc^n(G;\F_p)$ and $[d]\in \hc^m(G;\F_p)$ is given by
	\begin{align*}
		[c]\cup [d](x,y)&=[c\cup d](x,y)\in \hc^{n+m}(G;\F_p).
	\end{align*}
	For more on cup products, see for example \cite[Sec.\ I.5, Sec.\ V.3]{KSBrown}.
	We remark that the abelian group $\hc^*(G;\F_p)$ together with the cup product is a graded-commutative ring, equivalently, for $[c]\in \hc^n(G;\F_p)$ and $[d]\in \hc^m(G;\F_p),$ the equality $[c]\cup[d]=(-1)^{nm}[d]\cup[c]$ holds.

	In this paper we will work only with the first $\hc^1(G;\F_p)$ and the second $\hc^2(G;\F_p)$ cohomology groups. These have a group theoretic interpretation and are very well understood. Since the action of $G$ on $\F_p$ is trivial, we have $\hc^1(G; \F_p)=\Hom(G,\F_p)$. Moreover, there is a one-to-one correspondence between the cohomology classes $[c]\in \hc^2(G;\F_p)$ and the equivalence classes of (central) group extensions 
	\begin{equation}\label{eq:extensionofGbyFp}
		0 \longrightarrow \F_p \overset{\iota}\longrightarrow E\overset{\rho}\longrightarrow G\longrightarrow 0,
	\end{equation}
	where the equivalence is defined as follows. 
	For a group $H$, let $\id_H$ denote the identity map on $H$. Two group extensions $E$ and $E'$ are \emph{equivalent} if there exists an isomorphism $\varphi: E\to E'$ making the next diagram commutative
	\begin{align*}
		\xymatrix{
			0 \ar[r] &\F_p \ar[d]^{\id_{\F_p}} \ar[r] &E \ar[r] \ar[d]_{\varphi} & G\ar[r] \ar[d]^{\id_{G}} &0\\
			0 \ar[r] &\F_p \ar[r] &E' \ar[r] &G \ar[r]&0.
		}
	\end{align*}
	Following \cite[Sec.\ IV. 3]{KSBrown}, we outline the aforementioned correspondence. Given an extension $E$ of $G$ by $\F_p$ as in \eqref{eq:extensionofGbyFp}, we 
	fix a set-theoretic map $s: G\to E$ with $\rho\circ s=\id_{G}$ and define $c\in C^2(G;\F_p)$ to be the 2-cocycle such that, for each $g_1,g_2\in G$, the equality 
	\[
	\iota(c(g_1,g_2))=s(g_1)s(g_2)s(g_1+g_2)^{-1}
	\]
	is satisfied.
	It can be shown that $[c]$ does not depend on the choice of $s$. Similarly, given $[c]\in H^2(G;\F_p)$, we construct a group extension as in \eqref{eq:extensionofGbyFp}. For this, we choose a representative $c\in C^2(G;\F_p)$ and consider the set $E_c=  G\times\F_p$, endowed with the product
	\[
	(g_1,m_1) \cdot (g_2,m_2)=(g_1+g_2, m_1+m_2+c(g_1,g_2)).
	\]
	With the definition of
	\begin{align*}
		\iota:\F_p\longrightarrow E_c,& \quad m\longmapsto \iota(m)=(0,m), \\
		\rho: E_c\longrightarrow G, & \quad (g,a) \longmapsto \rho(g,a)=g,
	\end{align*}
	we get that  $ 0 \to \F_p \overset{\iota}\to E_c\overset{\rho}\to G\to 0$ is indeed a group extension.
	
	\subsection{Cohomology of abelian $p$-groups}
	\label{subsec:strongisoLongexact}
	
	We proceed by describing the cohomology ring structure for finite abelian $p$-groups. To that aim, we observe that the cohomology ring of the cyclic $p$-group $\Z/(p^k)$ of order $p^k$ is given, as a graded ring, by the following:
	\[
	\hc^*(\Z/(p^k);\F_p)\cong \Lambda(y)\otimes \F_p[x] \textup{ for } \begin{cases}
		k\geq 1 & \textup{ if } p>2,\\
		k>1 & \textup{ if } p=2. 
	\end{cases}
	\]
	Here 
	$\Lambda(\cdot)$ denotes the exterior algebra and the generators $[y], [x]\in \hc^*(\Z/(p^k);\F_p)$ are of degrees $|y|=1$ and $|x|=2$.
	Following the notation and assumptions in Section \ref{subsec:assumptions}, using the K\"unneth formula for cohomology \cite[Sec.\ 2.5]{LEvens} and the fact that $\F_p$ is a field, we obtain the following isomorphism of graded rings
	\begin{equation}\label{eq:CohomologyAbelianGroups}
		\hc^*(G;\F_p)\cong \Lambda(y_1, \dots, y_{r+1})\otimes \F_p[x_1, \dots, x_{r+1}], 
	\end{equation}
	where the generators $[y_i]$ and $[x_i]$ have degrees $|y_i|=1$ and $|x_i|=2$ for $i\in \{1, \dots, r+1\}$ (see \cite[Sec.\ V.6]{KSBrown}). Moreover, for every $i\in\{1, \dots, r+1\}$, the element $x_i$ can be chosen to be $\beta(y_i)$, where $\beta$ is an appropriate higher order Bockstein homomorphism; see \cite[Sec.\ 6.2, Proof of Thm.\ 6.21]{JMcCleary} and Section \ref{subsec:bockstein}.
	\begin{remark}\label{rmk:p=2case}
		If we allowed $p$ to be $2$ with not all $n_j$'s at least $2$, then the cohomology ring of $G$ would not be isomorphic to the tensor product in \eqref{eq:CohomologyAbelianGroups} anymore (see for instance \cite[Sec.\ 3.3]{LEvens}). For this reason, we excluded such cases from our study.
	\end{remark}
	
	As we have seen in Section \ref{subsec:cohomologyofgroups}, the elements of $\hc^2(G;\F_p)$ correspond to central extensions of $G$ by $\F_p$; we denote by $\hab(G;\F_p)$ the subset of those that correspond to abelian extensions. We remark that $\hab(G;\F_p)$ is in fact the abelian group $\ext^1_{\Z_pG}(G,\F_p)$ \cite[Thm.\ 3.4.3]{CAWeibel}, whose elements are the abelian extension classes of $\Z_pG$-modules with trivial $G$-action
	and whose operation is the Baer sum; for more detail, see for example \cite[Sec.\ 3.4]{CAWeibel}. 
	Moreover,  $\hc^2(G;\F_p)$ decomposes as a sum of the following  $\F_p$-vector spaces (see \cite[Sec.\ 3.4]{LEvens} or \cite[11.4.16 and 11.4.18]{DJSRobinson81}):
	\[
	\hc^2(G;\F_p)=\hab(G;\F_p)\oplus\langle\im \cup\rangle\cong\hab(G;\F_p)\oplus \langle y_i\cup y_j\;  : 1\leq i <j\leq r+1\rangle.
	\]
	In addition, we have that 
	$$\dim_{\F_p}\hab(G;\F_p)=r+1 \textup{ and } \dim_{\F_p} \Lambda^2(y_1, \dots,y_{r+1})= \binom{r+1}{2}$$
	so, in particular, if $G$ is cyclic, then $\hc^2(G;\F_p)=\hab(G;\F_p)$.
	\subsection{Pontryagin dual}\label{subsec:Pontryagin}
	
	We define here the Pontryagin dual of $G$ following \cite[Ch.\ 3]{SMorris} and stress that this notion can be extended to arbitrary locally compact groups.
	Let $\mathbb{T}=\{z\in \mathbb{C}: \; |z|=1\}$ denote the circle group.
	
	\begin{definition}\label{def:pontryagindual}
		The Pontryagin dual of $G$ is the abelian group $\widehat{G}=\Hom(G,\mathbb{T})$ consisting of all homomorphisms from $G$ to $\mathbb{T}$.
	\end{definition}
	
	Since the exponent of $G$ is $p^n$, each element of $\widehat{G}=\Hom(G,\mathbb{T})$ will have image contained in $\mathbb{T}[p^n]$, the $p^n$-th torsion subgroup of $\mathbb{T}$, which is cyclic of order $p^n$. Without loss of generality, we identify $\widehat{G}$ with $\Hom(G,C/(p^nC))$.
	As the $\Hom$ functor commutes with direct sums, we have that
	\begin{align*}
		\widehat{G}=\Hom(G,C/(p^nC))&=\Hom(\bigoplus_{j=1}^{t}I_{j},C/(p^nC))\cong \bigoplus_{j=1}^{t} \Hom( I_{j},C/(p^nC))=\bigoplus_{j=1}^{t} \widehat{I_{j}} \\
		&=\Hom(\bigoplus_{j=1}^{t} \bigoplus_{k=1}^{r_j} I_{jk},C/(p^nC))\cong \bigoplus_{j=1}^{t} \bigoplus_{k=1}^{r_j}\Hom( I_{jk},C/(p^nC))=\bigoplus_{j=1}^{t} \bigoplus_{k=1}^{r_j} \widehat{I_{jk}}\,;
	\end{align*}
	see for example also \cite[Thm.\ 13]{SMorris}. The last series of maps induces the following isomorphism 
	\begin{equation}\label{eq:psi2tilde}
		\widehat{\phi_1}: G\longrightarrow\widehat{G}=\Hom(G,C/(p^nC)), \quad  \gamma_{jk}\mapsto (\gamma_{jk}^*: \gamma_{ih}\longmapsto \delta_{(j,k),(i,h)}p^{n-n_j}\gamma),
	\end{equation}
	which generalizes the isomorphism $\phi_1:V\rightarrow \widehat{V}$ from Section \ref{subsec:assumptions}.
	
	\subsection{Higher Bockstein homomorphism}\label{subsec:bockstein}
	
	We give here another characterization of $\hab(G;\F_p)$.
	To this end, we let $A$ act on $\Hom(G,C)$ and on $\Hom(G, C/p^nC)$ as described in \eqref{eq:ActionAonHomToK} and
	observe that the natural short exact sequence
	\[0\longrightarrow C[p]{\longrightarrow} C\overset{\pi}{\longrightarrow} C/(p^nC)\longrightarrow 0,\]
	induces a long exact sequence of $\Z_pA$-modules
	\begin{equation}\nonumber
		\begin{tikzpicture}[descr/.style={fill=white,inner sep=1.5pt}]
			\hspace{-10pt}
			\matrix (m) [
			matrix of math nodes,
			row sep=1em,
			column sep=2.5em,
			text height=1.5ex, text depth=0.25ex
			]
			{ 0 & C[p] & C & C/(p^nC) & \\
				& \Hom(G,C[p]) & \Hom(G,C) & \Hom(G,C/(p^nC)) & \\
				& \hc^2(G;C[p]) & \hc^2(G;C) & \hc^2(G;C/(p^nC)) & \ldots, \\
			};
			\path[overlay,->, font=\scriptsize,>=latex]
			(m-1-1) edge (m-1-2)
			(m-1-2) edge node[descr,above=0.5ex] {} (m-1-3) 
			(m-1-3) edge node[descr,above=0.5ex] {} (m-1-4)
			(m-1-4) edge[out=355,in=175,red] (m-2-2)
			(m-2-2) edge (m-2-3)
			(m-2-3) edge node[descr,above=0.5ex] {$\pi_B$} (m-2-4)
			(m-2-4) edge[out=355,in=175,red] node[descr,yshift=0.3ex] {$\beta$} (m-3-2)
			(m-3-2) edge (m-3-3)
			(m-3-3) edge (m-3-4)
			(m-3-4) edge (m-3-5);
		\end{tikzpicture}
	\end{equation}
	where $\pi_B(f)=\pi \circ f$ and $\beta$ is the connecting homomorphism \cite[Sec.\ 1.3, Add.\ 1.3.3]{CAWeibel}. The homomorphism $\beta$ is classically known as the \emph{(higher order) Bockstein homomorphism}; see for instance \cite[Sec.~5.2]{JRHarper}, \cite[Sec.\ 6.2, p.\ 197]{JMcCleary}. We write $\im\pi_B=\pi_B(G)$ and stress that $\pi_B$ respects the action of $A$. 
	Observe, moreover, that $\hc^2(G;C[p])$ is naturally isomorphic to $\hc^2(G;\F_p)$ and so we identify them.

	\begin{lemma}\label{lem:ImBeta=Hab} 
		The image of $\beta$ is an $\F_p$-vector space of dimension $\dG(G)$ and
		the following equalities hold
		$$\im \beta =\hab(G;\F_p)= \ext_{\Z_pG}^1(G,\F_p).$$
	\end{lemma}
	
	\begin{proof}
		We start by showing that $\im\beta$ is contained in $\hab(G;\F_p)$. For this, let $[c]\in \im \beta$ and let $E_c$ be an extension of $G$ by $\F_p$ represented by $[c]$. 
		By definition of $\beta$, there exists a map $\tilde{c}: G\to C$ such that, for all $x,y\in G$, one has  $c(x,y)=\tilde{c}(x)+\tilde{c}(y)-\tilde{c}(x+y)$. Then, for all $g_1,g_2\in G$ and $m_1,m_2\in C[p]$, we have that
		\begin{align*}
			(g_1,m_1)\cdot (g_2,m_2)-(g_2,m_2)\cdot (g_1,m_1) & =(0,c(g_1,g_2)-c(g_2,g_1)) \\
			& =(0, -\tilde{c}(g_1+g_2)+\tilde{c}(g_2+g_1))=(0,0)
		\end{align*}
		and thus $E_c$ is abelian. This shows that $\im\beta\subseteq\hab(G;\F_p)$.
		Now we show that the following holds:
		\begin{equation}\label{eq: KerBeta}
			\ker\beta=\graffe{\sum_{j=1}^t\sum_{k=1}^{r_j}\alpha_{jk}\gamma^*_{jk} \mid \alpha_{jk}\in p\Z_p}=p\hat{G}.
		\end{equation}
		For this, note that $\beta$'s image is contained in the elementary abelian $p$-group $\hc^2(G;\F_p)$ and so it follows that $p\Hom(G, C/(p^nC))\subseteq \ker \beta$. We also have that
		\[
		\frac{\hat{G}}{p\hat{G}}=\frac{\Hom(G, C/(p^nC))}{p\Hom(G, C/(p^nC))}\cong \frac{G}{pG}
		\]
		and $\dim_{\F_p}\hab(G;\F_p)=r+1$ thus the first isomorphism theorem yields \eqref{eq: KerBeta}.
		It follows that $\widehat{\phi_1}$ 
		induces an isomorphism $V\rightarrow \widehat{G}/\ker\beta$ and $\im\beta$ is an $\F_p$-vector space of dimension $\dim_{\F_p}V=\dim_{\F_p}\hab(G;\F_p)$. We derive that $\im \beta=\hab(G;\F_p) = \ext_{\Z_pG}^1(G,\F_p)$.
	\end{proof}
	
	\noindent
	\begin{remark}
		Observe that $\im \beta =\hab(G;\F_p)= \ext_{\Z_pG}^1(G,\F_p)$ is precisely the collection of equivalence classes of symmetric 2-cocycles, i.e. cocycles $c$ with the property that, for all $g_1, g_2 \in G$, the equality $c(g_1,g_2)=c(g_2,g_1)$ holds. Since $A$ maps symmetric 2-cocycles to symmetric 2-cocycles, the action of $A$ on $\hc^2(G;\F_p)$ induces an action of $A$ on $\hab(G;\F_p)$. This can also be derived from Lemma \ref{lem:ImBeta=Hab}.
	\end{remark}
	
	\subsection{Maximal subgroups}\label{subsec:maximalsubgroups}
	
	In this section we describe a map that associates each cohomology class in $\hab(G;\F_p)$ to a subgroup of index at most $p$ in $G$. Recall that $\widehat{V}=\Hom(V,\F_p)$ denotes the dual of $V$ and $\widehat{G}=\Hom(G,C/(p^nC))$ the Pontryagin dual of $G$.
	Let $\phi: \widehat{G}\to \widehat{V}$ be the homomorphism defined by $\gamma_{jk}^*\mapsto v_{jk}^*$, in other words, $\phi=\phi_1\pi\widehat{\phi_1}^{-1}$. 
	It follows that
	$$\ker \phi=\{f\in\Hom(G,C/(p^nC))\mid  f(I_j)\subseteq (p^{n-n_j+1}C)/(p^nC), \, 1\leq j\leq t\}=p\Hom(G, C/(p^nC))=\ker \beta$$
	and thus $\phi$ induces an isomorphism 
	$$\phi_2: \widehat{G}/\ker \beta=\Hom(G, C/(p^n C))/\ker \beta \longrightarrow \widehat{V}=\Hom(V,\F_p).$$
	Let now $\phi_3$ be the isomorphism induced by $\beta$, i.e.
	$$\phi_3:\Hom(G,C/(p^nC))/\ker\beta\longrightarrow \im\beta=\hab(G;\F_p),$$
	where the fact that $\im\beta=\hab(G;\F_p)$ is ensured by Lemma \ref{lem:ImBeta=Hab}.
	Now, the map $\phi_4=\phi_3\circ\phi_2^{-1}$ is an isomorphism and we obtain the following commutative diagram:
	\begin{equation}\label{eq:commdiag}
		\xymatrix{
			& \widehat{V} \ar[d]^{\phi_4} &V \ar[l]_{\ {\phi_1}} &G \ar[l]_{\pi}  \\
			\widehat{G}/\ker\beta \ar[ur]^{{\phi_2}}\ar[r]^{ {\phi_3}\ } &\ \hab(G;\F_p).& &
		}
	\end{equation}
	
	\begin{lemma}\label{lem:phi4Ainv}
		The isomorphisms $\phi_2,\phi_3,\phi_4$ respect the action of $A$.
	\end{lemma}
	
	\begin{proof}
		Since $\phi_3$ clearly respects the action of $A$, it suffices to show that $\phi_2$ is an $A$-isomorphism on the generators $\gamma_{jk}^*\in \widehat{G}$ described in \eqref{eq:psi2tilde}. Let $(\sigma, \lambda)\in A$. Since $\sigma (\ker \beta)\subseteq \ker\beta$, the following equalities hold
		\begin{align*}
			\phi_2((\sigma,\lambda)(\gamma_{jk}^*+\ker\beta))&=\phi_2(\lambda \gamma_{jk}^*\sigma^{-1}+ \ker\beta)= \lambda v_{jk}^*\overline{\sigma}^{-1}=(\sigma,\lambda)\phi_2(\gamma_{jk}^*+\ker\beta)
		\end{align*}
		and thus both $\phi_2$ and $\phi_4=\phi_3\circ \phi_2^{-1}$ are $A$-isomorphisms.
	\end{proof}
	We rely on the commutative diagram \eqref{eq:commdiag} to define the following function 
	\begin{align*}
		\tau: \hab(G;\F_p) & \longrightarrow \graffe{\textup{subspaces of codimension at most $1$ of } V} \\\nonumber
		[c] &\longmapsto \ker(\phi_4^{-1}([c]))
	\end{align*}
	and remark that, by construction, 
	$\tau([c])=V$ if and only if $[c]=[0]$.
	In particular $\tau$ induces a bijection
	\[
	\mathfrak{t}_V:\mathbb{P}\hab(G;\F_p) \longrightarrow \graffe{\textup{hyperplanes of  }V},
	\]
	equivalently, postcomposing with $\pi^{-1}$, a bijection
	\begin{align}\label{eq:TauMap}
		\mathfrak{t}_G: \ & \mathbb{P}\hab(G;\F_p)\longrightarrow \graffe{\textup{subgroups of index $p$ of } G}. 
	\end{align}
	We will show in Section \ref{subsec:compatibility} that the maps $\mathfrak{t}_G$ and $\mathfrak{t}_V$ are compatible with the actions of $A$ as given in \eqref{eq:ActionAonSbgs} and \eqref{eq:ActionAonPH^2}; see Corollary \ref{cor:stabT}. In particular, it will follow that each non-trivial orbit of $A$ in $\hab(G;\F_p)$ has cardinality divisible by $p-1$. This is true in higher generality. We warn the reader that in the sequel we will often, by a slight abuse of notation, write $\mathfrak{t}_G([c])$ meaning the image under $\mathfrak{t}_G$ of the projective class of $[c]$.
	
	\begin{lemma}\label{lem:lambda=1}
		Let $\lambda\in\Z_p^*$, $[c]\in\hc^2(G;\F_p)$, and $[\omega]\in\gen{\im\cup}$. Then the following hold:
		\begin{enumerate}[label=$(\arabic*)$]
			\item $\lambda[c]=[c]$ if and only if $[c]=0$ or $\lambda=1$,
			\item if $[c]\in\hab(G;\F_p)$, then $(\lambda,\lambda)[c]=[c]$ and $(\lambda,\lambda)[\omega]=\lambda^{-1}[\omega]$.
		\end{enumerate}
		Moreover, every non-trivial orbit of $\hc^2(G;\F_p)$ has cardinality divisible by $p-1$.  
	\end{lemma}
	
	\begin{proof}
		(1) Suppose that $\lambda[c]=[c]$, i.e.\ there exists $f\in C^1(G;\F_p)$ such that for all $x,y\in G$,
		\[
		\lambda c(x,y)= c(x,y)+\partial_1(f)(x,y) \Longleftrightarrow (\lambda-1)c(x,y)=\partial_1(f)(x,y).
		\]
		If $\lambda \neq 1$, then, for all $x,y\in G$, it holds that $c(x,y)=(\lambda-1)^{-1}\partial_1(f)(x,y)$. Now define $\tilde{f}=(\lambda-1)^{-1}f$ to obtain that $c(x,y)=\partial_1(\tilde{f})(x,y)$ and thus $[c]=[0]$. The other implication is clear.
		
		(2) Let $\tilde{c}\in\Hom(G,C/(p^nC))$ be such that $[c]=\beta(\tilde{c})$ and note that $\tilde{c}$ exists by Lemma \ref{lem:ImBeta=Hab}. Let, moreover, $f,g\in\Hom(G,\F_p)$. Then, for each $x,y\in G$, we have
		\begin{align*}
			(\lambda,\lambda)[c]&=\beta(\lambda \tilde{c}\lambda^{-1})=\beta(\tilde{c})=[c],\\
			(\lambda,\lambda)(f\cup g)(x,y) & = \lambda f(\lambda^{-1}x)g(\lambda^{-1}y)=\lambda\lambda^{-2}f(x)g(y)=\lambda^{-1}(f\cup g)(x,y).
		\end{align*}
		We are now done since $\gen{\im\cup}$ is the linear span of elements of the form $[f\cup g]$.
	\end{proof}
	
	\begin{definition}\label{def:kerT,c-idx}
		Let $[c]\in\hab(G;\F_p)$ and let $M$ be a subgroup of $G$. Then the
		\begin{enumerate}[label=$(\arabic*)$]
			\item 
			\emph{kernel of $[c]$ in $G$} is the subgroup 
			$$T_c=\begin{cases}
				G & \textup{ if } [c]=0,\\
				\mathfrak{t}_G([c]) & \textup{ otherwise.}
			\end{cases}$$
			\item \emph{$c$-index of $M$} is the number
			$i_c(M)=\dim_{\F_p}((M+T_c)/T_c)\in \graffe{0,1}$.
		\end{enumerate}
	\end{definition}
	
	\noindent
	We will often write $T$ for $T_c$ if the cohomology class $[c]$ is clear from the context.
	
	\begin{ex}\label{ex:CB-ok}
		Let $j\in\graffe{1,\ldots, t}$ and $k\in\graffe{1,\ldots,r_j}$ and let $\beta$ be as in Section \ref{subsec:bockstein}.  
		Set $[c]=\beta(\gamma^*_{jk})$.
		Then $\phi_4^{-1}([c])=v_{jk}^*$ and it follows that the kernel of $[c]$ is 
		$
		T=\mathfrak{t}_G([c])=\pi^{-1}(\ker v_{jk}^* )=\ker(\gamma^*_{jk})+pG.
		$
	\end{ex}
	
	\subsection{Pl\"ucker embedding}\label{subsec:plucker}
	
	Recall the definition of the cup product 
	$\cup: \hc^1(G;\F_p)\times \hc^1(G;\F_p) \longrightarrow \hc^2(G;\F_p)$ as given in Section \ref{subsec:cohomologyofgroups}. In the present section, we construct maps on $\im\cup$ that will allow us to interpret $\im\cup$ as a specific family of subgroups of $G$. This construction is based on the Pl\"ucker embedding for Grassmannians; see for example \cite[Sec.\ 1.24]{ISafarevich},\cite[Ch.\ 5]{InvitationNLA}. 
	Until the end of the current section, for each positive integer $k$ and $\F_p$-vector space $W$, we denote by $\cor{G}(k,W)$ the Grassmannian of $k$-dimensional linear subspaces of $W$. Denote, moreover, by $\wedge$ the exterior product map $W\times W\rightarrow \Lambda^2 W$.
	
	We start by remarking that the vector spaces $\langle \im \cup \rangle$ and $\Lambda^2\widehat{V}$ are naturally isomorphic. The cup product being bilinear and alternating, the universal property of wedge products yields the surjective homomorphism
	\[
	\psi_G:\Lambda^2\Hom(G,\F_p)\longrightarrow \gen{\im\cup} \textup{ satisfying } f\wedge g \longmapsto [f\cup g].
	\]
	Observe that the last map is our announced isomorphism, since $\Hom(G,\F_p)$ and $\Hom(V,\F_p)=\widehat{V}$ are naturally isomorphic and the dimensions of  $\Lambda^2\widehat{V}$ and $\gen{\im\cup}$ are the same. Moreover, by its definition, $\psi_G$ satisfies $\psi_G(\im\wedge)=\im\cup$ and thus induces a bijection $\mathbb{P}\im\wedge\rightarrow\mathbb{P}\im\cup$. 
	
	We proceed by describing the Pl\"ucker embedding
	$s: \cor{G}(2,\widehat{V})\rightarrow \mathbb{P}(\Lambda^2 \widehat{V})$.
	For each $2$-dimensional subspace $U$ of $\widehat{V}$, fix an $\F_p$-basis $(f_u,g_u)$ of $U$ and define $s(U)=[f_u \wedge g_u]$. It is not difficult to show that $s$ is well-defined and that its image is equal to $\mathbb{P}\im\wedge$. We use now the map $s$ to define a bijection $\cor{G}(\dG(G)-2,V)  \rightarrow \mathbb{P}\im\cup$. For that, we start by identifying $\cor{G}(2,\widehat{V})$ and $\cor{G}(\dG(G)-2,V)$ via
	\[
	\cor{G}(2,\widehat{V})\longrightarrow \cor{G}(\dG(G)-2,V), \quad U=\F_pf_u\oplus\F_pg_u \longmapsto \ker f_u \cap \ker g_u.
	\]
	Composing maps in the obvious way, we get the following well-defined bijection
	\begin{equation}\nonumber
		\mathfrak{m}_V: \mathbb{P}\im\cup  \longrightarrow  \cor{G}(\dG(G)-2,V), \quad [\omega]=[f\cup g] \longmapsto \mathfrak{m}_V([\omega])=\pi(\ker f\cap \ker g),
	\end{equation}
	inducing the bijection
	\begin{equation}\label{eq:correspondenceM}
		\mathfrak{m}_G: \mathbb{P}\im\cup  \longrightarrow  \graffe{\pi^{-1}(U) \mid U \in \cor{G}(\dG(G)-2,V)}, \quad [\omega]=[f\cup g] \longmapsto \mathfrak{m}_G([\omega])=\ker f\cap \ker g.
	\end{equation}
	The last map identifies each element of $\mathbb{P}\im\cup$ with a subgroup $M$ of $G$ of index $p^2$ that contains $pG$. We next show that $\mathfrak{m}_G$ respects the action of $A$. As for the case of $\mathfrak{t}_G$ we will slightly abuse notation writing $\mathfrak{m}_G([\omega])$ for the image of the projective class of $[\omega]$ under $\mathfrak{m}_G$.
	
	\begin{lemma}\label{lemma:kernels}
		Let $[\omega]\in\mathbb{P}\im\cup$ and $(\sigma,\lambda)\in A$. 
		Then the equality $\sigma(\mathfrak{m}_G([\omega]))=\mathfrak{m}_G((\sigma,\lambda)[\omega])$ holds.
	\end{lemma}
	
	\begin{proof}
		Write $[\omega]=[f\cup g]$ with $f,g\in\Hom(G;\F_p)$. Then, for each choice of $x,y\in G$, we have
		\[
		(\sigma,\lambda)(f\cup g)(x,y)=\lambda (f\cup g)(\sigma^{-1}(x),\sigma^{-1}(y))=\lambda f(\sigma^{-1}(x))g(\sigma^{-1}(y)).
		\]
		In other words, $(\sigma,\lambda)(f\cup g)=(\lambda f\sigma^{-1})\cup(g\sigma^{-1})$ and we derive that
		\[
		\mathfrak{m}_G((\sigma,\lambda)[\omega])=\ker(\lambda f\sigma^{-1})\cap \ker(g\sigma^{-1})=\sigma (\ker f)\cap \sigma(\ker g)=\sigma(\ker f \cap \ker g)=\sigma(\mathfrak{m}_G([\omega])).
		\]
		This concludes the proof. 
	\end{proof}
	
	\begin{corollary}\label{cor:mmap}
		The map 
		$$\mathfrak{m}_G:\mathbb{P}\im\cup\rightarrow \graffe{M \textup{ subgroup of } G \textup{ with } G/M\textup{  elementary abelian of rank } 2}$$
		is a bijection respecting the action of $A$. 
	\end{corollary}
	
	\begin{definition}\label{def:kerM,c-idx}
		Let $[c]\in\hab(G;\F_p)$ and $[\omega]\in\im\cup$. 
		Then the
		\begin{enumerate}[label=$(\arabic*)$]
			\item 
			\emph{kernel of $[\omega]$ in $G$} is the subgroup 
			$$M_{\omega}=\begin{cases}
				G & \textup{ if } [\omega]=0,\\
				\mathfrak{m}_G([\omega]) & \textup{ otherwise.}
			\end{cases}$$
			\item \emph{$c$-index of $[\omega]$} is the number
			$i_c([\omega])=i_c(M_{\omega})$.
		\end{enumerate}
	\end{definition}
	
	We remark that, with the notation from Definition \ref{def:kerM,c-idx}, it is not difficult to show that, if $E$ is an extension defined by $[\omega]$, then the image of $\ZG(E)$ in $G$ coincides with $M_{\omega}$.
	
	\begin{lemma}\label{lemma:T=kerf'}
		Let $[\omega]\in\im\cup$ and let $M=M_{\omega}$ be the kernel of $[\omega]$ in $G$.
		Let $M\subseteq H,K\subseteq G$ be distinct maximal subgroups of $G$. 
		Then there exist $f,g\in\Hom(G,\F_p)$ such that $H=\ker f$, $K=\ker g$, and $[\omega]=[f\cup g]$. 
	\end{lemma}
	
	\begin{proof}
		Any maximal subgroup can be written as the kernel of a homomorphism $G\rightarrow\F_p$. Now, $H$ and $K$ being distinct, the claim follows from the fact that the map $\mathfrak{m}_G$ from \eqref{eq:correspondenceM} is well-defined.
	\end{proof}

	\section{Subgroup levels and compatibility}\label{sec:generalresults}
	
	We recall briefly the notation introduced in Section \ref{subsec:assumptions} that will be relevant here. If two elements $[c],[d]\in\hc^2(G;\F_p)$ belong to the same $A$-orbit, we will write $[c]\sim_A[d]$. For $[c]\in\hc^2(G;\F_p)$, we will write $A_c$ meaning the stabilizer of $[c]$ in $A$ and, if two elements $[d],[e]$ are in the same orbit under the induced action by $A_c$, we will write $[d]\sim_{A_c}[e]$. For a subgroup $K$  of $G$, we denote by $A_K$ the stabilizer of $K$ in $A$ with respect to the action from \eqref{eq:ActionAonSbgs}. 
	
	\subsection{Subgroup levels}
	The aim of this section is to introduce subgroup levels and prove basic properties about them. Subgroup levels are the key objects allowing us to describe the $A$-orbits on $\hab(G;\F_p)\times \im\cup$ combinatorially. 
	
	\begin{definition}\label{def:lLH}
		Let $M$ and $T$ be subgroups of $G$. Then the \emph{$T$-levels of $M$} are the entries of the pair $\ell\L_T(M)=(\ell_T(M), \L_T(M))$ where
		\begin{enumerate}[label=$(\arabic*)$]
			\item $\ell_T(M)=1+\max\graffe{0\leq i\leq \log_p\exp(T)\, :\, T[p^i]\subseteq M\cap T}$,
			\item $\L_T(M)=\min\graffe{j\in\Z_{\geq 0}\, :\, T[p^j]+(M\cap T)=T}$.
		\end{enumerate}
		If $T=G$, simply write $\ell\L(M)=(\ell(M),\L(M))$ for $\lL_G(M)$.
	\end{definition}
	
	\begin{ex}\label{ex:index0}~
		\begin{enumerate}[label=$(\arabic*)$]
			\item The $G$-levels vector of $G$ is $(n+1,0)$. More generally, the $T$-levels vector of $G$ is $(\log_p\exp(T)+1, 0)$.
			\item Assume that 
			$
			G=\Z/(p^2)\oplus\Z/(p^2)\oplus\Z/(p^3)\oplus\Z/(p^3)=\langle\gamma_{11}, \gamma_{12}, \gamma_{21}, \gamma_{22}\rangle,
			$
			and define 
			\begin{align*}
				T  &= \langle \gamma_{12}, \gamma_{21}, \gamma_{22} \rangle+ pG=\graffe{(pt_1,t_2,t_3,t_4) \mid t_i\in\Z_p}\subseteq G, \\ 
				M &= \langle \gamma_{12}, \gamma_{21}-\gamma_{22} \rangle +pG =\graffe{(pm_1,m_2, m_3,-m_3+pm_4)\mid m_i\in\Z_p}\subseteq G.
			\end{align*}
			It follows that 
			$\lL(M)=(2,3)$ and $\lL_T(M)=(3,3)$. 
			Since $T$ is a maximal subgroup of $G$, we can associate to it an element $[c]\in\hab(G;\F_p)\setminus\graffe{0}$ via \eqref{eq:TauMap}. For such a cohomology class $[c]$, the $c$-index of $M$ is $i_c(M)=0$, because $M$ is contained in $T$ (see Definition \ref{def:kerT,c-idx}). 
		\end{enumerate}
	\end{ex}
	
	\noindent
	We generalize the last example in the form of the following proposition.
	
	\begin{proposition}
		Let $X$ be a proper subgroup of $G$ containing $pG$ and with $|G:X|=p^k$. Let 
		\[
		p^{\alpha_1}\geq\ldots\geq p^{\alpha_{r+1}}
		\ \ \textup{ and } \ \ 
		p^{\beta_1}\geq\ldots\geq p^{\beta_{r+1}}
		\] 
		denote the elementary divisors of $G$ and $X$, respectively. Then there exist indices $r+1\geq i_1>\ldots>i_k\geq 1$ such that the following holds:
		\[
		\beta_i-\alpha_i =\begin{cases}
			1 & \textup{ if } i\in\{i_1,\ldots,i_k\}, \\
			0 & \textup{ otherwise.}
		\end{cases}
		\]
		Moreover, one has $\ell(X)=\alpha_{i_k}$ and $\L(X)=\alpha_{i_1}$.
	\end{proposition}
	
	\begin{proof}
		Since $X$ contains $pG$ and $|G:X|=p^k$, it is clear that the sequence $r+1\geq i_1>\ldots>i_k\geq 1$ of integers exists. From the definition of the $i_j$'s it is now easy to see that $\ell(X)=\alpha_{i_k}$ and $\L(X)=\alpha_{i_1}$.
	\end{proof}
	
	\begin{corollary}\label{cor:levMax}
		If $T$ is a maximal subgroup of $G$, then $\ell(T)=\L(T)$.
	\end{corollary}
	
	\begin{corollary}\label{cor:lLiso}
		Let $X,Y$ be subgroups of $G$ containing $pG$ and satisfying $|G:X|=|G:Y|\leq p^2$. 
		Then the following are equivalent:
		\begin{enumerate}[label=$(\arabic*)$]
			\item $X$ and $Y$ are isomorphic,
			\item $\lL(X)=\lL(Y)$.
		\end{enumerate}
	\end{corollary}
	
	\noindent
	It would be interesting to generalize the concept of $G$-levels in order to get a version of Corollary \ref{cor:lLiso} holding for all subgroups of $G$ containing $pG$.
	
	\begin{definition}\label{def:lLfunctions}
		Let $[c],[d]$ be elements of $\hab(G;\F_p)\cup\im\cup$ and let $K$ and $H$ be the kernels of $[c]$ resp.\ $[d]$ in $G$.
		Let $T$ be a subgroup of $G$. Then the $T$-levels and $[d]$-levels of $[c]$ are defined respectively as
		$$\lL_T([c])=\lL_T(K) \textup{ and } \lL_d([c])=\lL_H(K).$$
		If $T=G$ or $[d]=0$, simply write $\lL([c])=(\ell([c]),\L([c]))$ for $\lL_G([c])$ and $\lL_0([c])$.
	\end{definition}
	
	\noindent
	Below, we give some properties of $T$-levels.

	\begin{lemma}\label{lemma:lLbasic}
		Let $M$ and $T$ be subgroups of $G$. 
		Then the following hold:
		\begin{enumerate}[label=$(\arabic*)$]
			\item $1\leq \ell_T(M)$ and $\L_T(M)\leq \log_p\exp(T)$;
			\item if $z\in T\setminus (M\cap T)$, then $|z|\geq p^{\ell_T(M)}$; 
			\item if $\L_T(M)=\ell_T(M)-1$, then $T=\graffe{0}$;
			\item if $T\not\subseteq M$, then $\ell_T(M)-1<\L_T(M) $.
		\end{enumerate}
	\end{lemma}
	
	\begin{proof}
		Set $l_T=\ell_T(M)$ and $L_T=\L_T(M)$. 
		Items (1)-(2) are straightforward. To prove (3)-(4), we start by observing that, when $T$ is contained in $M$, then $l_T=\log_p\exp(T)+1$ and $L_T=0$. Assume now that $M$ does not contain $T$. Then $T\neq\graffe{0}$, $L_T\neq 0$, and, since $T[p^{l_T-1}]$ is contained in $M$ but $T[p^{L_T}]$ is not, we have that $l_T-1\lneq L_T$.
	\end{proof}

	\begin{lemma}\label{lemma:levelsTlevels}\label{lemma:GTLevels}
		Let $M$ and $T$ be subgroups of $G$ such that $|G:T|=p$ and $|G:M|=p^2$.
		Define $\lL(M)=(l,L)$ and $\lL_T(M)=(l_T,L_T)$. Then the following inequalities hold:
		\[
		l\leq l_T\leq \min\graffe{ L, L_T}=\begin{cases}
			L & \textup{ if } M \not\subseteq T,\\
			L_T & \textup{ if } M \subseteq T.
		\end{cases}
		\]
	\end{lemma}

	\begin{proof}
		The inequality $l\leq l_T$ follows from the fact that, for each $m\in\Z_{\geq 0}$, one has 
		\[
		G[p^m]\subseteq M \Longrightarrow T[p^m]=G[p^m]\cap T\subseteq M\cap T.
		\] 
		We now show that $l_T\leq L$. For a contradiction, assume that $l_T>L$. It follows from the definition of $\lL_T(M)$ that $T[p^L]$ is contained in $M\cap T$. In particular, since $T[p^L]=G[p^{L}]\cap T$, we have that $T[p^L]$ is contained in $M\cap G[p^L]$. Now, since $|G:T|=p$, we get that $|G[p^L]:T[p^L]|\leq p$ and consequently
		\begin{align*}
			p^2=|G:M|&=|(M+G[p^L]):M|=|G[p^L]:(M\cap G[p^L])|\leq |G[p^L]:T[p^L]|\leq p
		\end{align*}
		providing a contradiction. So we have proven that $l_T\leq L$. The inequality $l_T\leq L_T$ follows from Lemma \ref{lemma:lLbasic} and thus yields that $l_T\leq \min\graffe{L,L_T}$.
		
		For the last equality, assume first that $M$ is contained in $T$.
		Since $G=M+G[p^L]$, we have that $T\cap G=T=M+T[p^{L}]$. By definition of $L_T$, we have $L_T\leq L$.
		To conclude, assume that $M$ is not contained in $T$. From
		\[
		T=(M\cap T)+T[p^{L_T}]=(M+T[p^{L_T}])\cap T
		\]
		we get that $G=M+T[p^{L_T}]=M+G[p^{L_T}]$. It follows from the minimality of $L$ that $L\leq L_T$.
	\end{proof}

\clearpage
	\begin{ex}\label{ex:index1}~
		\begin{enumerate}[label=$(\arabic*)$]
			\item Assume that $G$ is a free $\Z/(p^3)$-module of rank $2$. Then $M=G[p^2]=pG$ has index $p^2$ in $G$ and is contained in every maximal subgroup of $G$. For each $T$ maximal in $G$, it then holds that $3=\ell(M)=\ell_T(M)=\L_T(M)=\L(M)$.
			\item Let 
			$
			G=\Z/(p)\oplus\Z/(p^2)\oplus\Z/(p^3)\oplus\Z/(p^4)=\langle \gamma_{11}, \gamma_{21}, \gamma_{31}, \gamma_{41} \rangle,
			$ and define
			\begin{align*}
				T&=
				\gen{\gamma_{21}, \gamma_{31}, \gamma_{41}}=\graffe{(0,t_2,t_3,t_4) \mid t_i\in\Z_p}\subseteq G, \\
				M&=\gen{\gamma_{11}-\gamma_{31}-\gamma_{41},\gamma_{21}}+pG=\graffe{(m_1,m_2,-m_1+pm_3,-m_1+pm_4) \mid m_i\in\Z_p}\subseteq G.
			\end{align*}
			It follows that $M\cap T=\gen{\gamma_{21}}+pG$ and therefore $1=\ell(M)<3=\ell_T(M)=\L(M)<4=\L_T(M)$. We conclude by observing that, since $M$ is not contained in $T$, the $c$-index of $M$ in $G$ will be $1$ for each $[c]\in\hab(G;\F_p)$ that realizes $T$ in the sense of \eqref{eq:TauMap}.
		\end{enumerate}
	\end{ex}

	\subsection{Compatibility}\label{subsec:compatibility}
	
	The following proposition collects a number of basic properties shared by elements belonging to the same $A$-orbit.

	\begin{proposition}\label{prop:LevelsTindices}
		Let $[c],[d]\in\hab(G;\F_p)$ and let $T_c$ and $T_d$ denote respectively the kernels of $[c]$ and $[d]$.  Let, moreover, $a=(\sigma,\lambda)\in A$ be such that $[d]=a\cdot [c]$. Let $[\omega]=[f\cup g]\in\im\cup$ and let $M=M_{\omega}$ be the kernel of $[\omega]$. Define
		\[
		f_a=\lambda f \sigma^{-1} \textup{ and } g_a=g\sigma^{-1}
		\]
		and let $M_a$ be the kernel of $[\omega_a]=[f_a\cup g_a]$. Then 
		$a\cdot ([c]+[\omega])=[d]+[\omega_a]$ and the following hold:
		\begin{enumerate}[label=$(\arabic*)$]
			\item The following maps are inverses to each other:
			\begin{align*}
				\phi: \im\cup/A_c \longrightarrow \im\cup/A_{d}, & \quad  A_c[\omega] \longmapsto A_{d}(a\cdot[\omega]), \\
				\psi: \im\cup/A_{d} \longmapsto \im\cup/A_c 
				& \quad A_{d}[\omega] \mapsto A_c(a^{-1}\cdot[\omega]). 
			\end{align*}
			\item $T_d=\sigma(T_c)$ and $M_a=\sigma(M)$.
			\item $\lL(M)=\lL(M_a)$, $\lL_c(M)=\lL_d(M_a)$, $\lL(T_c)=\lL(T_d)$ and $i_c(M)=i_d(M_a)$.
		\end{enumerate}
	\end{proposition}
	
	\begin{proof}
		To show that $a\cdot ([c]+[f\cup g])=[d]+[f_a\cup g_a]$ is an easy computation. 
		
		(1) Straightforward. 
		
		(2) That $\sigma(M)=M_a$ is a straight consequence of Lemma \ref{lemma:kernels}. We prove that $\sigma(T_d)=T_c$.
		We use the bar notation for the subspaces of $V=G/pG$ and we refer to the notation in \eqref{eq:commdiag} and \eqref{eq:ActionAonV}.
		The map $\sigma$ being an isomorphism, it follows from Lemma \ref{lem:phi4Ainv} that
		\begin{align*}
			\overline{T_d} & = \ker(\phi_4^{-1}([d])) =\ker(\phi_4^{-1}(a\cdot [c]))\\
			& = \ker(a\cdot\phi_4^{-1}([c])) = \ker(\lambda \phi_4^{-1}([c])\circ\overline{\sigma}^{-1}) \\
			& = \ker(\phi_4^{-1}([c])\circ \overline{\sigma}^{-1})= \overline{\sigma} \ker(\phi_4^{-1}([c]))\\
			& = \overline{\sigma}(\overline{T_c}).
		\end{align*}
		Lifting everything back to $G$, we get $T_d=\sigma(T_c)$.
		
		(3) This is a direct consequence of (2) and Definitions \ref{def:kerT,c-idx} and \ref{def:kerM,c-idx}.
	\end{proof}
	
	\noindent
	In the next result, let $\mathfrak{t}_G$ and $\mathfrak{m}_G$ denote respectively the maps from \eqref{eq:TauMap} and \eqref{eq:correspondenceM}. For each $k\in\graffe{1,\ldots,d}$, we write moreover
	\begin{align*}
		\cor{S}_G^{(k)} & = \graffe{ \pi^{-1}(W) \mid W \textup{ subspace of codimension } k \textup{ of } V} \\
		& = \graffe{ K \textup{ subgroup of } G \textup{ with } G/K \textup{ elementary abelian of rank } k}
	\end{align*}
	and note that the action of $A$ given in \eqref{eq:ActionAonSbgs} naturally induces an action of $A$ on each $\cor{S}_G^{(k)}$ and, component-wise, on any of their products.
	
	\begin{corollary}\label{cor:stabT}
		The following is an isomorphism of $A$-sets:
		\[\mathfrak{(t,m)}_G:\mathbb{P}\hab(G;\F_p)\times \mathbb{P}\im\cup \longrightarrow \cor{S}_G^{(1)}\times \cor{S}_G^{(2)}, \quad ([c],[\omega])\longmapsto (\mathfrak{t}_G([c]),\mathfrak{m}_G([\omega])).\]
		Moreover, for each $[c]\in\hab(G;\F_p)\setminus\graffe{0}$ with kernel $T$, the stabilizer $A_c$ is a subgroup of $A_T$ 
		of index $|A_T:A_{c}|=p-1$.
	\end{corollary}
	
	\begin{proof}
		The map $\mathfrak{(t,m)}_G$ is an isomorphism of $A$-sets as a consequence of Corollary \ref{cor:mmap} and Proposition \ref{prop:LevelsTindices}(2).
		Therefore, we get that, for each element $[c]\in \hab(G;F_p)$, if $T=\mathfrak{t}_G([c])$, then $A_c\subseteq A_T$ and $|A_T:A_c|=p-1$.
	\end{proof}
	
	\noindent
	We point out the connection between Corollary \ref{cor:stabT} and Lemma \ref{lem:lambda=1}. The last corollary clearly describes the projective nature of the orbits in terms of subgroups of $G$. It would be interesting to know whether the map $\mathfrak{(t,m)}_G$ can be extended to the whole of $\mathbb{P}\hc^2(G;\F_p)$; see also Section \ref{subsec:moregens}.

	\section{Abelian extensions}\label{section AbelianExtensions}
	
	In this section we classify the $A$-orbits of $\hab(G;\F_p)$ via classifying the $A$-orbits 
	in $\Hom(G,C/p^nC)/\ker\beta$, where $\beta$ is the homomorphism introduced in Section \ref{subsec:bockstein}.
	We also show that, under our assumptions, strong isomorphism classes and isomorphism types of extensions of $G$ by $\F_p$ coincide. 
	
	Until the end of Section \ref{section AbelianExtensions}, the following assumptions will hold. For $j\in\graffe{1,\ldots, t}$ and $k\in\graffe{1,\ldots, n_j}$, let $\gamma_{jk}^*$ be the dual of $\gamma_{jk}$ as defined in \eqref{eq:psi2tilde}, within Section \ref{subsec:Pontryagin}. 
	For $j\in\graffe{1,\ldots,t}$, define moreover $\gamma^*_j=\gamma_{j1}^*$ and denote by $\widehat{\pi}_j$ the natural projection $\widehat{G}=\bigoplus_{j=1}^t\widehat{I_j}\rightarrow \widehat{I_j}$.
	The next proposition is the main result of the current section.
	
	\begin{proposition}\label{prop:abelian}
		Let $[c],[d]\in\hab(G;\F_p)$. Then the following are equivalent:
		\begin{enumerate}[label=$(\arabic*)$]
			\item $[c]\sim_A [d]$;
			\item $[c]=[d]=0$ or there exists $i\in\graffe{1,\ldots, t}$ such that $[c]\sim_A\beta(\gamma_i^*)\sim_A[d]$;
			\item $\lL([c])=\lL([d])$;
			\item $\ell([c])=\ell([d])$;
			\item $\L([c])=\L([d])$.
		\end{enumerate} 
	\end{proposition}

	\subsection{Bockquivalence relation}\label{subsec:Bockquivalence}
	
	In this section we prove Proposition \ref{prop:abelian} via studying the action of $A$ on $\widehat{G}=\Hom(G,C/(p^nC))$. We recall from Section \ref{subsec:bockstein} that, since $\beta$ respects the action of $A$, Lemma \ref{lem:ImBeta=Hab} yields that the $A$-orbits of $\hab(G;\F_p)$ are in natural bijection with the $A$-orbits of $\widehat{G}/\ker\beta$.
	
	\begin{definition}[Bockquivalence relation]
		Two elements $f,g\in \widehat{G}$ are \emph{Bockquivalent}, written $f\approx_G g$, if there exist $(\sigma,\lambda)\in A$ and $\varepsilon\in \pi_B(G)$ such that $g=\lambda f\sigma^{-1}+\varepsilon$. 
	\end{definition} 
	
	\noindent
	The just defined Bockquivalence relation is clearly an equivalence relation, because it describes the $A$-orbits of  $\Hom(G,C/(p^nC))/\ker\beta$.
	We will refer to the corresponding equivalence classes as \emph{Bockquivalence classes} and, if $f\in \Hom(G, C/(p^n C))$, we will write $\llbracket f\rrbracket$ to denote the Bockquivalence class of $f$. Our immediate goal is to determine representatives for the Bockquivalence classes of $G$.
	
	\begin{proposition}\label{prop representatives}
		Let $\Gamma=\graffe{\gamma^*_j : j=1,\ldots,t}\cup\graffe{0}$. Then $\Gamma$ is a set of representatives for the Bockquivalence classes of $G$ and, for each $j\in\graffe{1,\ldots,t}$, the following equality holds:
		\[
		\llbracket \gamma^*_j\rrbracket=\graffe{c\in\widehat{G} : |\im\widehat{\pi}_j(c)|=p^{n_j},|\im\widehat{\pi}_l(c)|<p^{n_l} \textup{ for }  l>j}.
		\]
		Moreover, $G$ has exactly $t+1$ Bockquivalence classes. 
	\end{proposition}
	
	\begin{proof}
		We start by recalling that $\ker\beta=p\hat{G}$, as given in \eqref{eq: KerBeta}. We will show that the images of the maps $\gamma_i^*$ in the quotient $\hat{G}/p\hat{G}=\hat{G}/\ker\beta$ constitute a set of representatives for the nonzero orbits of the action of $A$ on the last quotient.
		We first show that each nonzero orbit can be represented by one of the $\gamma_i^*$'s. To this end,  for every nontrivial orbit choose a representative  $f\in\hat{G}\setminus p\hat{G}$ of the form
		\[
		f=\sum_{j=1}^t\sum_{k=1}^{r_j}\alpha_{jk}\gamma_{jk}^*, \textup{ with } \alpha_{jk}\in\Z_p^*\cup\{0\}
		\]
		and let $i\in\{1,\ldots,t\}$ be maximal such that there exists $s\in\{1,\ldots,r_i\}$ with $\alpha_{is}\in\Z_p^*$. 
		It follows from the maximality of $i$ that $f(G)$ is generated by $p^{n-n_i}\gamma$ and so the first  isomorphism theorem yields that $G=\gen{\gamma_{is}}\oplus\ker f$. Set now $H=\gen{\gamma_{jk} \mid (j,k)\neq (i,1)}$ and note that $G=\gen{\gamma_{i1}}\oplus H$. Since $\gamma_{i1}$ and $\gamma_{is}$ have the same order, the elementary divisor theorem yields an automorphism $\sigma$ of $G$ sending $\gamma_{is}$ to $\gamma_{i1}$ and $\ker f$ to $H$. As a consequence, $(\sigma,1)f=\gamma_i^*$ and $f$ is in the orbit of $\gamma^{*}_i$. 
		
		We now show that any two $\gamma_i^*$'s represent distinct orbits. To this end, let $i\geq j$ be such that $\gamma_i^*$ and $\gamma_j^*$ represent the same $A$-orbit in $\hat{G}/p\hat{G}$ and let $(\sigma,\lambda)\in A$ and $g\in\hat{G}$ be such that $\gamma_i^*=(\sigma,\lambda)\gamma_j^*+pg$. It follows that 
		\[
		p^{n-n_i}\gamma=\gamma_{i}^*(\gamma_{i1})=(\sigma,\lambda)\gamma_j^*(\gamma_{i1})+pg(\gamma_{i1})
		\]
		and so, by taking orders, we derive that $n_i=\max\{|\gamma_j^*(\sigma^{-1}(\gamma_{i_1}))|,n_i-1\}\leq \max\{n_j,n_i-1\}$. From the fact that $i\geq j$, that is $n_i\geq n_j$, we conclude that $i=j$.
	\end{proof} 
	
	\begin{proof}[Proof of Proposition \ref{prop:abelian}]
		
		$(1)\Leftrightarrow(2)$ This is Proposition \ref{prop representatives}. 
		
		$(2)\Leftrightarrow(3)\Leftrightarrow(4)\Leftrightarrow(5)$ 
		Thanks to Proposition \ref{prop representatives}, a set of representatives of the $A$-orbits of $\hab(G;\F_p)$ is given by
		$0,\beta(\gamma_1^*), \ldots, \beta(\gamma_t^*)$.
		As a consequence of Example \ref{ex:CB-ok} and Proposition \ref{prop:LevelsTindices}(2), the $A$-orbits are uniquely determined by their $G$-levels, which 
		are respectively
		$(n+1,0), (n_1,n_1), \ldots, (n_t,n_t)$.
	\end{proof}

	\subsection{Convenient orbit representatives}\label{sec:ReductionProblem}
	
	The goal of this section is to produce, for each given $[c]\in\hab(G;\F_p)$, a representative of the $A$-orbit of $[c]$ that can be conveniently expressed in terms of the choice of generators we made in Section \ref{subsec:assumptions} and is thus more suitable to computations. We essentially want to be able to regard elements of $\hab(G;\F_p)$ as if they were images of the generators of $\widehat{G}$.
	
	Let $[c]\in\hab(G;\F_p)$ and let $\tilde{c}$ be an element of $\widehat{G}$ such that $[c]=\beta(\tilde{c})$; recall that $\tilde{c}$ exists thanks to Lemma \ref{lem:ImBeta=Hab}. 
	Then, thanks to Proposition \ref{prop:abelian}, there exists $a\in A$ and $$b\in \cor{B}=\graffe{\gamma_{jk} \mid 1\leq j\leq t, 1\leq k\leq n_j}$$ such that, for $\graffe{b_1,\ldots, b_r}=\cor{B}\setminus\graffe{b}$, the following hold
	\[
	\ker(a\cdot\tilde{c})=\bigoplus_{i=1}^r\gen{b_i} \textup{ and } 
	G=\gen{b}\oplus \ker(a\cdot\tilde{c})=\gen{b}\oplus\bigoplus_{i=1}^r\gen{b_i}.
	\]
	Set $\tilde{d}=a\cdot\tilde{c}$ and $[d]=\beta(\tilde{d})=[a\cdot c]$. Let, moreover, $T_c$ and $T_d$ denote the kernels of respectively $[c]$ and $[d]$. Then, thanks to Example \ref{ex:CB-ok}, we know that $T_d=\ker\tilde{d}+pG$ and so we have a very concrete description of $T_d$ in terms of the elements of $\cor{B}$. Moreover, if we are interested in the action of $A_c$ on $\im\cup$, we can as well consider the action of $A_d$ on $\im\cup$, thanks to Proposition \ref{prop:LevelsTindices}(1).

	\subsection{Strong isomorphism}\label{subsec:strongiso}

	We close Section \ref{section AbelianExtensions} by showing that strong isomorphism classes of $G$ by $\F_p$ coincide with isomorphism classes of extensions of $G$ by $\F_p$.
	
	\begin{proposition}\label{prop:SIC=IC}
		Let $E_c$ and $E_d$ be central extensions of $G$ by $\F_p$ represented by the cohomology classes $[c]$ and $[d]$ in $\hc^2(G;\F_p)$, respectively. Then, $E_c$ and $E_d$ are isomorphic if and only if $[c]\sim_A [d]$. 
	\end{proposition}
	
	\begin{proof}
		If $c\sim_A d$, then, thanks to Theorem \ref{th orbits and strong exts}, the extensions $E_c$ and $E_d$ are strongly isomorphic, so in particular isomorphic. Assume now that $E_c$ and $E_d$ are isomorphic. If $E_c$ is nonabelian, then $[E_c,E_c]$ has order $p$ and is mapped, by any isomorphism $E_c\rightarrow E_d$, to $[E_d,E_d]$. So, $E_c$ and $E_d$ are strongly isomorphic and we are done by Theorem \ref{th orbits and strong exts}. 
		We conclude by observing that each isomorphism class of extensions of $G$ by $\F_p$ is a union of strong isomorphism classes. It is well-known that there are $t+1$ possible isomorphism types of abelian extensions of $G$ by $\F_p$ and now, thanks to Proposition \ref{prop representatives}, we know that there are exactly $t+1$ strong isomorphism classes of such extensions. As the numbers are the same, we are done.
	\end{proof}

	\section{Nonabelian extensions}\label{sec:stabilizers}
	
	Let $[c]\in\hab(G;\F_p)$ and denote by $A_c$ the stabilizer of $[c]$ in $A$. 
	The aim of this section is to determine the orbits of the action of $A_c$ on the image of the cup product 
	$\cup:\Hom(G,\F_p)\times\Hom(G,\F_p)\rightarrow\hc^2(G;\F_p).$
	We will prove the following result.
	
	\begin{proposition}\label{prop:imcup}
		Let $[c]\in\hab(G;\F_p)$ and $[\omega]$, $[\vartheta]$ be elements of $\im\cup$. The following are equivalent:
		\begin{enumerate}[label=$(\arabic*)$]
			\item $[\omega]\sim_{A_c}[\vartheta]$,
			\item $(\lL([\omega]),\lL_c([\omega]),i_c([\omega]))=(\lL([\vartheta]),\lL_c([\vartheta]),i_c([\vartheta]))$.
		\end{enumerate}
	\end{proposition}
	
	Until the end of Section \ref{sec:stabilizers}, the following assumptions will be satisfied. Let $[c]\in\hab(G;\F_p)$ and $[\omega],[\vartheta]\in\im\cup$ be fixed. As a consequence of the discussion from Section \ref{sec:ReductionProblem}, without loss of generality, we will work under the following additional assumptions. 
	Let $\cor{B}=\graffe{b_0=b,b_1,\ldots,b_r}$ be a minimal set of generators  of cardinality $r+1$ such that
	\[
	G= \langle b\rangle \oplus \bigoplus_{i=1}^r \langle b_i\rangle.
	\]
	Let $\tilde{c}\in\widehat{G}$ be such that $[c]=\beta(\tilde{c})$ and, if $[c]\neq 0$, assume that 
	$\im \tilde{c}\cong \langle b \rangle$ and that 
	$$\ker\tilde{c}=  \bigoplus_{i=1}^r \langle b_i\rangle \textup{ and } G=\gen{b}\oplus\ker\tilde{c}.$$
	Let $T$ be the kernel of $[c]$ and, if $[c]\neq 0$, observe that $T=\ker \tilde{c}+pG$ is maximal in $G$, analogously to Example \ref{ex:CB-ok}. Write, moreover, $M_{\omega}$ and $M_{\vartheta}$ respectively for the kernels of $[\omega]$ and $[\vartheta]$, respectively.
	The case $[c]=0$ is covered in Section \ref{subsec:CupUnderA}. If $[c]\in \hab(G;\F_p)\setminus\graffe{0}$, then we study the action of $A_c$ on cup products in two parts. 
	The case where $M_{\omega}+M_{\vartheta}\subseteq T$ 	is discussed in Section \ref{subsec:ic=0} and the case where $G=M_{\omega}+T=M_{\vartheta}+ T$ 
	is considered in Section \ref{subsec:ic=1}. 
	We remark that, the condition $i_c([w])=i_c([\vartheta])$ imposed in $(2)$ prevents the existence of any other case.
	We last let $M$ be a subgroup of index $p^2$ of $G$ containing $pG$ and 
	observe that $M$ is the kernel of some element of $\im\cup\setminus\graffe{0}$; see Section \ref{subsec:plucker}.

	\begin{lemma}\label{lemma:lLstart}
		Write $\ell\L(M)=(l,L)$. Let, moreover, $\tilde{M}$ be a subgroup of $M$ and $\cor{C}\subseteq \cor{B}$ such that $G=\gen{\cor{C}}\oplus\tilde{M}$. 
		Then there exist $x,y\in\cor{C}$ such that $|x|=p^l$, $|y|=p^L$, and $G=\gen{x,y}+M$. 
	\end{lemma}
	
	\begin{proof}
		We start by showing that there exists $x\in\cor{C}$ such that $|x|=p^l$ and $x\notin M$.
		For a contradiction, assume this is not true and write $C=\gen{\cor{C}}$.  
		Then
		$G[p^l]=C[p^l]+\tilde{M}[p^l]\subseteq C[p^{l-1}]+M=M$,
		which is a contradiction to the maximality of $l$. 
		Fix now such an element $x$ and define $\tilde{H}=\gen{x}\oplus \tilde{M}$, which satisfies $G=\gen{\cor{C}\setminus\{x\}}\oplus\tilde{H}$. 
		Note that $\tilde{H}$ is a subgroup of the maximal subgroup $H=\gen{x}+M$ of $G$.
		We now claim that there exists $y\in\cor{C}\setminus\{x\}$ of order $p^L$.  If this is not the case and $D=\gen{\cor{C}\setminus\{x\}}$,
		then 
		\[
		G[p^L]=D[p^L]+\tilde{H}[p^L]\subseteq D[p^{L-1}]+H\subseteq G[p^{L-1}]+H
		\]
		from which it follows that
		\[
		G=G[p^L]+M=G[p^{L-1}]+H=G[p^{L-1}]+\gen{x}+M.
		\]
		The minimality of $L$ yields that $l=L$ and so that $G=\gen{x}+M$. In particular, $|G:M|=|\gen{x}:\gen{px}|=p$. Contradiction.
	\end{proof}
	
	\begin{theorem}\label{th:gustavo}
		Write $\lL(M)=(l,L)$ and let $x,y\in\cor{B}$ be such that $G=\gen{x,y}+M$ and $(|x|,|y|)=(p^l,p^L)$. Let, moreover, $H$ be a subgroup of $G$ such that $x,y\in H$. Then there exists a subgroup $\tilde{M}\subseteq H\cap M$ such that $H=\gen{x}\oplus\gen{y}\oplus \tilde{M}$.
	\end{theorem}
	
	\begin{proof}
		Let $J$ be a subgroup of $G$ such that $G=\gen{x}\oplus\gen{y}\oplus J$ and note that $J$ exists because $x,y\in\cor{B}$. Moreover, thanks to Dedekind's Law, we also have that $H=\gen{x}\oplus\gen{y}\oplus(H\cap J)$. Write now $I=\gen{x}\oplus\gen{y}$. We will show that $H\cap J$ can be replaced by a complement of $I$ in $H$ that is contained in $M$. For this, we consider all decompositions of $H$ of the form
		\[
		H=I\oplus\gen{z_1}\oplus\ldots\oplus\gen{z_s}
		\]
		and we choose one such that $m=|\{i \mid z_i \notin M\}|$ is minimal. We will prove that $m=0$, in other words that $C=\gen{z_1}\oplus\ldots\oplus\gen{z_s}$ is the desired complement. We argue by contradiction, assuming that $z_1\notin M$. It follows that $|z_1|\geq p^l$ and, from $G=\gen{x,y}+M$ and $pG\subseteq M$, that $z_1$ can be expressed as 
		\begin{equation}\label{eq:pI}
			z_1=\eta x+\kappa y + z_1' \ \ \textup{ with } \ \ \eta, \kappa\in\{0,\ldots, p-1\},\ z_1'\in H\cap M.
		\end{equation}
		We claim that $C'=\gen{z_1'}\oplus\gen{z_2}\oplus\ldots\oplus\gen{z_s}$ is a complement of $I$ in $H$. We will show this by means of proving that $I\oplus\gen{z_1}=I\oplus\gen{z_1'}$. Since the equality $I+\gen{z_1}=I+\gen{z_1'}$ is clear, it suffices to verify that $I\cap \gen{z_1'}=0$ holds. For this, let $\lambda,\mu,\nu\in\Z_p$ be such that 
		$\lambda x+ \mu y+\nu z_1'=0$. It follows that 
		\[
		(\lambda-\nu\eta)x+(\mu-\nu\kappa)y+\nu z_1=0, 
		\]
		from which we derive that $(\lambda-\nu\eta)x=(\mu-\nu\kappa)y=\nu z_1=0$. Then $\nu\geq |z_1|\geq p^l$ and, the order of $x$ being $p^l$ yields that 
		$0=(\lambda-\nu\eta)x=\lambda x$. If, additionally $|z_1|\geq p^L$ or $\kappa=0$, in a similar fashion we obtain that $\mu y=0$. We assume now that $|z_1|<p^L$ and that $\kappa \neq 0$. Then $|x|$ is also smaller than $p^L$. Moreover, $\kappa$ is invertible modulo $p$ and so \eqref{eq:pI} yields that $y$ belongs to $\gen{x,z_1}+M$. We deduce that 
		\[
		G=\gen{x,y}+M=\gen{x,z_1}+M=G[p^{L-1}]+M,
		\]
		which contradicts the definition of $L=\L(M)$. This concludes the proof that $I\oplus\gen{z_1}=I\oplus\gen{z_1'}$.
		
		We have shown that $C'$ is a complement of $I$ in $H$ with a smaller number of generators outside of $M$; contradiction to the minimality of $m$.
	\end{proof}

	\subsection{Full stabilizer}\label{subsec:CupUnderA}
	
	Until the end of Section \ref{subsec:CupUnderA}, we work under the assumption that $[c]=[0]\in\hab(G;\F_p)$; then $A=A_c$ and we are simply studying the action of $A$ on the cup product. In this section we prove thus Proposition \ref{prop:imcup} under these assumptions and in the following form.
	
	\begin{proposition}\label{prop:lL}
		One has $[\omega]\sim_A [\vartheta]$ if and only $\lL([\omega])=\lL([\vartheta]).$
	\end{proposition}
	
	\noindent
	To that aim, we prove the following lemma, which will be used in the next section, too.

	\begin{lemma}\label{lemma:xyA-orbits}
		Write $\lL(M)=(l,L)$. Then there exist $f,g\in\Hom(G,\F_p)$, 
		$x,y\in\cor{B}$ of orders respectively $p^l$ and $p^L$,  and $\tilde{M}\subseteq M$
		such that the following hold:
		\begin{enumerate}[label=$(\arabic*)$]
			\item $M=\ker f\cap \ker g$,
			\item $f(x)=1$, $g(x)=0$, $f(y)=0$, and $g(y)=1$,
			\item $G=\gen{x}\oplus\gen{y}\oplus \tilde{M}.$
		\end{enumerate}
	\end{lemma}
	
	\begin{proof}
		Let $x,y$ be as in Lemma \ref{lemma:lLstart}, where $\cor{C}$ is taken to be $\cor{B}$. Now (1)-(2) are direct consequences of Lemma~\ref{lemma:T=kerf'} while (3) follows from Theorem~\ref{th:gustavo} to $H=G$. 
	\end{proof}

	\begin{proof}[Proof of Proposition \ref{prop:lL}]
		Assume first that $[\omega]\sim_A[\vartheta]$. If $[\omega]=[\vartheta]=0$, then
		we are clearly done. If $[\omega],[\vartheta]$ are non-trivial elements of $\im\cup$, then Proposition \ref{prop:LevelsTindices}(3) yields that $\lL(M_{\omega})=\lL(M_{\vartheta})$. 
		
		For the other implication, we start by observing that $\lL([\omega])=(n+1,0)$ if and only if $M_{\omega}=G$. In particular, the trivial class is determined by its $G$-levels.
		We assume now that $[\omega],[\vartheta]$ are non-trivial and write  $\lL([\omega])=\lL([\vartheta])=(l,L)$.
		We will construct $(\sigma, \lambda)\in A$ such that $[\vartheta]=(\sigma,\lambda)[\omega]$. To this end, 
		we let $x_{\omega},y_{\omega}\in G$, $f_{\omega},g_{\omega}\in\Hom(G;\F_p)$, and  $\tilde{M}_{\omega}\leq M_{\omega}$ be equivalents of $x,y,f,g,\tilde{M}$ in Lemma \ref{lemma:xyA-orbits} for $M_{\omega}$.
		Analogously, we let
		$x_{\vartheta},y_{\vartheta},f_{\vartheta},g_{\vartheta}$, and $\tilde{M}_{\vartheta}$ be associated with $M_{\vartheta}$. Observe that 
		$[\omega]=[f_{\omega}\cup g_{\omega}]$ and $[\vartheta]=[f_{\vartheta}\cup g_{\vartheta}]$.
		We now choose an isomorphism $\tilde{M}_{\omega}\rightarrow \tilde{M}_{\vartheta}$ and extend it to an automorphism  
		$\sigma\in \Aut(G)$ satisfying 
		$\sigma(x_{\omega})= x_{\vartheta}$ and  $\sigma(y_{\omega})= y_{\vartheta}$.
		It is now a straightforward calculation to show that $(\sigma,1)[\omega]=[\vartheta]$.
	\end{proof}

	\subsection{Inclusion of the kernels}\label{subsec:ic=0}
	
	\noindent
	Until the end of Section \ref{subsec:ic=0}, we work under the assumption that $[c]\neq 0$; then $T=\ker \tilde c+pG$ is maximal in $G$. 
	We additionally assume that $M+M_{\omega}+M_{\vartheta}\subseteq T$ and observe that 
	$[\omega],[\vartheta]\neq 0$ and $i_c([\omega])=i_c([\vartheta])=0$. In this section we prove Proposition \ref{prop:ic=0}, which coincides with Proposition \ref{prop:imcup} under the last assumptions.
	
	\begin{proposition}\label{prop:ic=0}
		One has $[\omega]\sim_{A_c}[\vartheta]$ if and only if $\lL([\omega])=\lL([\vartheta])$.
	\end{proposition}
	
	\noindent
	The next result explains why the values $\lL_c([\omega])$ and $\lL_c([\vartheta])$ do not appear in Proposition \ref{prop:ic=0}.
	
	\begin{proposition}\label{prop:i=0,levels}
		Write $\lL(M)=(l,L)$ and $\lL_T(M)=(l_c,L_c)$.
		Then the following hold:
		\begin{enumerate}[label=$(\arabic*)$]
			\item $G[p^l]\subseteq T$ is equivalent to $l=l_c=L_c<L$,
			\item $G[p^l]\not\subseteq T$ is equivalent to $l\leq l_c=L_c=L$.
		\end{enumerate}
	\end{proposition}
	
	\begin{proof}
		By Lemma \ref{lemma:levelsTlevels},  we have $l\leq l_c\leq L_c\leq L$ and, since $T$ is maximal, Corollary~\ref{cor:levMax} yields 
		$l_c=L_c$.
		
		(1) Assume, for a start, that $G[p^l]\subseteq T$. Since $G[p^l]$ is not contained in $M$, we have that
		\[
		T=G[p^l]+M=T[p^l]+M=T[p^l]+(M\cap T)
		\]
		so the minimality of $L_c$ yields $l=L_c$.
		Moreover, since 
		$G=G[p^L]+M$, we also have that $L>l$. 
		
		Assume now that $l=l_c=L_c<L$ and, for a contradiction, that $G[p^l]$ is not contained in $T$. We then have that 
		$$G=G[p^l]+T=G[p^l]+M+T[p^{L_c}]=G[p^l]+M,$$
		contradicting the minimality of $L$.
		
		(2) Assume first that $G[p^l]$ is not contained in $T$.  Then we have
		$$
		G=T+G[p^l]=M+T[p^{L_c}]+G[p^l]=M+G[p^{L_c}]
		$$
		and so the minimality of $L$ yields $L=L_c$.
		The other implication follows from (1).
	\end{proof}
	
	\noindent
	The rest of the section is devoted to proving Proposition \ref{prop:ic=0}. 
	
	\begin{lemma}\label{lem:i=0,overview M}
		Write $\lL(M)=(l,L)$. Then there exist 
		$y\in\cor{B}$ and $\tilde{M}\subseteq M$
		such that $G=\gen{b}\oplus\gen{y}\oplus \tilde{M}$ and 
		$$(|b|,|y|)=\begin{cases}
			(p^l, p^L) & \textup{ if } G[p^l]\not \subseteq T,\\
			(p^L,p^l) & \textup{ if } G[p^l]\subseteq T.
		\end{cases}
		$$
	\end{lemma}
	
	\begin{proof}
		Let $x$, $y$, and $\tilde{M}$ be as in Lemma \ref{lemma:xyA-orbits}: since $T/M$ is cyclic of order $p$, we have that $x=b$ or $y=b$. 
		By renaming $y$ to be the element of $\graffe{x,y}$ that is not equal to $b$, we get  the claim. 
	\end{proof}

	\begin{proof}[Proof of Proposition \ref{prop:ic=0}]
		The implication from left to right follows in a straightforward way from Proposition \ref{prop:LevelsTindices}.
		We now show that the other direction also holds true.
		For this, write $\lL([\omega])=\lL([\vartheta])=(l,L)$.
		Let $(y_{\omega}, \tilde{M}_{\omega})$ and $(y_{\vartheta},\tilde{M}_{\vartheta})$ be the equivalents of the pair $(y,\tilde{M})$ from Lemma \ref{lem:i=0,overview M} respectively for $M_{\omega}$ and $M_{\vartheta}$. It follows that $|y_{\omega}|=|y_{\vartheta}|$ and $\tilde M_{\omega}\cong \tilde M_{\vartheta}$. We now let $\sigma\in\Aut(G)$ be such that 
		\[
		\sigma(b)=b, \quad \sigma(y_{\omega})=y_{\vartheta}, \quad
		\sigma(\tilde{M}_{\omega})=\tilde{M}_{\vartheta}.
		\]
		By construction, $(\sigma,1)$ stabilizes $T$ and satisfies $(\sigma,1)\cdot M_{\vartheta}=M_{\omega}$. Let $\lambda\in\Z_p^*$ be such that $(\sigma,\lambda)\in A_c$, the existence of $\lambda$ being guaranteed by Corollary \ref{cor:stabT}. Set $a=(\sigma,\lambda)$. Then we have that $a\in A_c$ satisfies $a(T,M_{\vartheta})=(T,M_{\omega})$ and thus, as a consequence of Corollary \ref{cor:stabT}, the elements $[\omega]$ and $[\vartheta]$ are conjugate under $A_c$ up to a scalar. Lemma \ref{lem:lambda=1}(2) yields the claim.
	\end{proof}

	\subsection{Incomparable kernels}\label{subsec:ic=1}
	
	\noindent
	Until the end of Section \ref{subsec:ic=1}, we work under the following additional assumptions. Assume that $[c]\neq 0$ and thus that $T=\ker \tilde c+pG$ is a maximal subgroup of $G$. 
	We assume, moreover, that $M,M_{\omega},M_{\vartheta}$ are not contained in $T$ and that 
	$[\omega],[\vartheta]\neq 0$. In particular, we have that $i_c([\omega])=i_c([\vartheta])=1$
	and that $G=M+T=M_{\omega}+T=M_{\vartheta}+T$. The goal of the present section is to prove Proposition \ref{prop:ic=1}, which coincides with Proposition \ref{prop:imcup} under the last assumptions.
	
	\begin{proposition}\label{prop:ic=1}
		One has $[\omega]\sim_{A_c}[\vartheta]$ if and only if $(\lL([\omega]), \lL_c([\omega]))=(\lL([\vartheta]),\lL_c([\vartheta]))$.
	\end{proposition}
	
	\noindent
	The proof of Proposition \ref{prop:ic=1} is divided into cases depending on the relations between $G$-levels and $T$-levels. 
	
	\begin{lemma}\label{lemma:xTyT}
		Write $\lL(M)=(l,L)$ and $\lL_T(M)=(l_c,L_c)$. Then there exist $f,g\in \Hom(G,\F_p)$, $x,y\in\graffe{b_1,\ldots,b_r}$ of orders respectively $p^{l_c}$ and $p^{L_c}$, and $\tilde{M}\subseteq \ker \tilde{c}\cap M$
		such that the following hold:
		\begin{enumerate}[label=$(\arabic*)$]
			\item $M=\ker f\cap \ker g$,
			\item $f(x)=1$, $g(x)=0$, $f(y)=0$, and $g(y)=1$,
			\item $G=\gen{b}\oplus\gen{x}\oplus\gen{y}\oplus {\tilde{M}}$.
		\end{enumerate}
		Moreover, there exist two distinct elements in $\graffe{b,x,y}$ of orders respectively $p^l$ and $p^L$.
	\end{lemma}
	
	\begin{proof}
		Let $x,y$ be as in Lemma \ref{lemma:lLstart}, where $\cor{C}$ is taken to be $\cor{B}$. Then (1) and (2) follow directly from Lemma~\ref{lemma:T=kerf'}. We now prove (3). To this end, define $\cor{B}'=\{pb,b_1,\ldots,b_{r+1}\}$ and let $M'=M\cap T$, which has index $p^2$ in $T$ and contains $pG$. Then, with $T$, $\cor{B}'$ and $M'$ in the roles of $G$, $\cor{C}$ and $M$, Lemma \ref{lemma:lLstart} yields $x,y\in\cor{B'}$ such that $T=\gen{x,y}+M'$ and $(|x|,|y|)=(p^{l_c},p^{L_c})$. Since $pb\in M'$, we derive that $x,y\in\cor{B}\setminus\{b\}$ and in particular $x,y\in\ker\tilde{c}$. Now applying Theorem~\ref{th:gustavo} to $T$, $M'$ and $H=\ker\tilde{c}$, we get a subgroup $\tilde{M}\subseteq M'\cap \ker \tilde{c}$ such that $\ker\tilde{c}=\gen{x}\oplus\gen{y}\oplus \tilde{M}$. Thanks to Lemma~\ref{lemma:lLstart}, two elements out of $\cor{C}=\{b,x,y\}$ have orders $p^l$ and $p^L$ and so we are done.
	\end{proof}
	
	\noindent
	Recall that, by Lemma \ref{lemma:levelsTlevels}, we have that $\ell(M)\leq \ell_T(M)\leq \L(M)\leq \L_T(M)$ and so, from the last result, we derive the following corollary in a straightforward way.
	
	\begin{corollary}
		One has $\ell(M)=\ell_T(M)$ or $\ell_T(M)=\L(M)$ or $\L(M)=\L_T(M)$. 
	\end{corollary}

	\noindent
	Until the end of Section \ref{subsec:ic=1}, we let $x$, $y$, and $\tilde{M}$ be as in Lemma \ref{lemma:xTyT}. We also write $\lL(M)=(l,L)$ and $\lL_T(M)=(l_c,L_c)$.
	
	\begin{lemma}\label{lemma:bM}
		There exist $\alpha,\delta\in\Z_p$ such that $b_M=b-\alpha x-\delta y\in M\setminus (M\cap T)$ and
		\[
		(\alpha,\delta)\in\begin{cases}
			\Z_p\times \Z_p & \textup{if } l=l_c\leq L=L_c, \\
			\Z_p^*\times \graffe{0} & \textup{if } l<l_c<L=L_c, \\
			\Z_p\times \Z_p^* & \textup{otherwise}.
		\end{cases}
		\]
		Moreover, if $l=l_c\leq L=L_c$, then $b_M$ and $b$ have the same order.
	\end{lemma}
	
	\begin{proof}
		We start by recalling that $G=\gen{x,y}+M$ and $M$ contains $pG$ and has index $p^2$ in $G$. As a consequence there exist uniquely determined $\alpha,\delta\in\{0,\ldots,p-1\}$ and $b_M\in M$ with the property that $b=\alpha x+\delta y+b_M$. Fix such triple and note that $b_M\notin M\cap T$ because $b\notin T$ while $x,y\in T$. We will prove the following:
		\begin{enumerate}[label=$(\roman*)$]
			\item if $l=l_c\leq L=L_c$, then $|b|=|b_M|$,
			\item if $l<l_c<L=L_c$, then $\alpha\neq 0$ and $\delta=0$,
			\item in all other cases $\delta\neq 0$.
		\end{enumerate} 
		We start by assuming that $l=l_c\leq L=L_c$. If $|b|\geq p^L$, then clearly $|b|=|b_M|$ and, if $|b|<p^l$, then $b\in M$ and thus again $|b|=|b_M|$. We assume in conclusion that $p^l\leq |b|<p^L$. In this case $\delta=0$ because otherwise $y\in \gen{b,x}+M$ yielding to the contradiction $G=\gen{b,x}+M=G[p^{L-1}]+M$. Since $\delta=0$, we readily derive $|b|=|b_M|$. 
		
		Assume now that $l<l_c<L=L_c$. As one of $b,x,y$ has order $p^l$, we have that $|b|=p^l$. Thus, if $\delta$ were nonzero, we would get a similar contradiction as the one from the previous case. Note that, $\delta$ being zero, $\alpha$ can't be otherwise we would have $b\in M$. This would yield a contradiction because, in such case, we would have that 
		\[
		G[p^l]=\gen{x^{p^{l_c-l}}}\oplus\gen{y^{p^{L_c-l}}}\oplus\gen{b}\oplus\tilde{M}[p^{l}]\subseteq pG+M=M,
		\]
		contradicting the minimality of $l$.
		
		We conclude by looking at the remaining cases. 
		Assume first that $L<L_c$. Since two of the elements $b,x,y$ have order $p^l$ and $p^L$, we have $|b|, |x|\leq p^L<p^{L_c}$. If, for a contradiction, $\delta$ were zero, we would have $|b_M|<p^{L_c}$ and consequently
		\[
		G=M+G[p^L]=(M\cap T)+\gen{b_M}+G[p^L]=(M\cap T)+G[p^{L_c-1}].
		\]
		In particular, this would imply that $T=(M\cap T)+T[p^{L_c-1}]$, contradicting the definition of $L_c$.
		We are now left with considering the case $l<l_c=L=L_c$. It follows from Lemma \ref{lemma:xTyT} that $|b|=p^l$ and, in particular, $b$ is not contained in $M$.
		Now, the elements $x$ and $y$ having the same orders, we assume without loss of generality that $\delta$ is invertible.
	\end{proof}

	\begin{lemma}\label{lemma:split}
		Assume that $(\lL([\omega]),\lL_c([\omega]))=(\lL([\vartheta]),\lL_c([\vartheta]))=(l,L,l_c,L_c)$ and, additionally, that $l=l_c\leq L=L_c$.
		Then one has $[\omega]\sim_{A_c}[\vartheta]$. 
	\end{lemma}
	
	\begin{proof}
		Let $f_{\omega},g_{\omega},f_{\vartheta},g_{\vartheta}$ play the roles of $f$ and $g$ from Lemma \ref{lemma:xTyT} respectively for $M_{\omega}$ and $M_{\vartheta}$. Let, analogously
		$x_\omega,y_\omega,x_\vartheta,y_\vartheta\in\ker\tilde{c}$ play the roles of $x$ and $y$ and let moreover $\tilde{M}_{\omega}$ and $\tilde{M}_{\vartheta}$ play the roles of $\tilde{M}$. 
		Write $b_{\omega}$ and $b_{\vartheta}$ for the equivalents of $b_M$, which we know have the same order thanks to the case $l=l_c\leq L=L_c$ in Lemma \ref{lemma:bM}. We have that $$G=\gen{b_\omega}\oplus\gen{x_\omega}\oplus\gen{y_\omega}\oplus\tilde{M}_\omega=\gen{b_\vartheta}\oplus\gen{x_\vartheta}\oplus\gen{y_\vartheta}\oplus\tilde{M}_\vartheta.$$
		Let now $\lambda\in\Z_p$ be such that $\tilde{c}(b_{\vartheta})=\lambda\tilde{c}(b_\omega)$ and note that such $\lambda$ exists by the definition of $b_M$. Let, moreover, $\sigma:G\rightarrow G$ be an isomorphism satisfying
		\[
		x_\omega\mapsto x_\vartheta, \quad y_\omega\mapsto \lambda y_\vartheta, \quad b_{\omega}\mapsto b_{\vartheta}, \quad \sigma(\tilde{M}_{\omega})={\tilde{M}_{\vartheta}}. 
		\]
		By construction, $a=(\sigma,\lambda)$ belongs to $A_c$ and satisfies $a[\omega]=[\vartheta]$.
	\end{proof}

	\begin{lemma}\label{lemma:(NS)l_T<L}
		Assume that $(\lL([\omega]),\lL_c([\omega]))=(\lL([\vartheta]),\lL_c([\vartheta]))=(l,L,l_c,L_c)$ and, additionally, that $l<l_c< L=L_c$.
		Then one has $[\omega]\sim_{A_c}[\vartheta]$. 
	\end{lemma}
	
	\begin{proof}
		Let $f_{\omega},g_{\omega},f_{\vartheta},g_{\vartheta}$ play the roles of $f$ and $g$ from Lemma \ref{lemma:xTyT} respectively for $M_{\omega}$ and $M_{\vartheta}$. Let, analogously
		$x_\omega,y_\omega,x_\vartheta,y_\vartheta\in\ker\tilde{c}$ play the roles of $x$ and $y$ and let moreover $\tilde{M}_{\omega}$ and $\tilde{M}_{\vartheta}$ play the roles of $\tilde{M}$. 
		Write $b_{\omega}=b-\alpha_{\omega}x_{\omega}$ and $b_{\vartheta}=b-\alpha_{\vartheta}x_{\vartheta}$ for the equivalents of $b_M$ from Lemma \ref{lemma:bM}; then $b_{\omega}\in\ker g_\omega$ and $b_\vartheta\in\ker g_\vartheta$.
		Let now $\lambda=\alpha_\vartheta\alpha_\omega^{-1}$ and let $\sigma:G\rightarrow G$ be an isomorphism satisfying
		\[
		x_\omega\mapsto \lambda x_\vartheta, \quad y_\omega\mapsto \lambda^{-1} y_\vartheta, \quad b\mapsto b, \quad \sigma({\tilde{M}_{\omega}})={\tilde{M}_{\vartheta}}. 
		\]
		By construction we have $(\sigma,1)\tilde{c}=\tilde{c}$ and $(\sigma,1)[\omega]=[\vartheta]$. 
	\end{proof}

	\begin{lemma}\label{lemma:(NS)twocases}
		Assume that $(\lL([\omega]),\lL_c([\omega]))=(\lL([\vartheta]),\lL_c([\vartheta]))=(l,L,l_c,L_c)$ and, additionally, that $l<l_c= L=L_c$ or $L<L_c$.
		Then one has $[\omega]\sim_{A_c}[\vartheta]$. 
	\end{lemma}
	
	\begin{proof}
		Let $f_{\omega},g_{\omega},f_{\vartheta},g_{\vartheta}$ play the roles of $f$ and $g$ from Lemma \ref{lemma:xTyT} respectively for $M_{\omega}$ and $M_{\vartheta}$. Let, analogously
		$x_\omega,y_\omega,x_\vartheta,y_\vartheta\in\ker\tilde{c}$ play the roles of $x$ and $y$ and let moreover $\tilde{M}_{\omega}$ and $\tilde{M}_{\vartheta}$ play the roles of $\tilde{M}$. 
		Write $b_{\omega}=b-\alpha_{\omega}x_\omega-\delta_\omega y_\omega$ and $b_{\vartheta}=b-\alpha_{\vartheta}x_\vartheta-\delta_\vartheta y_\vartheta$ for the equivalents of $b_M$ from Lemma \ref{lemma:bM}.
		Let now $\lambda=\delta_\vartheta\delta_\omega^{-1}$ and let $\sigma:G\rightarrow G$ be an isomorphism satisfying
		\[
		x_\omega\mapsto x_\vartheta, \quad y_\omega\mapsto \lambda y_\vartheta -\delta_\omega^{-1}(\alpha_\omega - \alpha_\vartheta)x_\vartheta, \quad b\mapsto b, \quad \sigma({\tilde{M}_\omega})={\tilde{M}_{\vartheta}}. 
		\]
		We start by observing that by construction $(\sigma,1)\tilde{c}=\tilde{c}$; moreover, $\sigma(M_{\omega})=M_\vartheta$ and $\sigma(T)=T$.  
		It follows from Corollary \ref{cor:stabT} that, up to a scalar, the elements $[\omega]$
		and $[\vartheta]$ are conjugate under $A_c$. Lemma \ref{lem:lambda=1}(2) yields the claim.
	\end{proof}
	
	\begin{proof}[Proof of Proposition \ref{prop:ic=1}]
		The implication from left to right follows in a straightforward way from Proposition \ref{prop:LevelsTindices}. We show the opposite one holds, too. Assume that $\lL([\omega])=\lL([\vartheta])=(l,L)$ and $\lL_c([\omega])=\lL_c([\vartheta])=(l_c,L_c)$. By Lemma \ref{lemma:levelsTlevels} we have that $l\leq l_c\leq L\leq L_c$. In case $(l,L)=(l_c,L_c)$, we are done by Lemma \ref{lemma:split}. Morover, if $l<l_c< L$, then we apply Lemma \ref{lemma:(NS)l_T<L}.
		The leftover cases are $L<L_c$ and $l<l_c=L=L_c$, which we resolve using Lemma \ref{lemma:(NS)twocases}.
	\end{proof}

	\begin{proof}[Proof of Proposition \ref{prop:imcup}]
		The implication $(1)\Rightarrow(2)$ is given by Proposition \ref{prop:LevelsTindices}(3). We now prove that $(2)\Rightarrow(1)$. For this, we assume that $(\lL([\omega]),\lL_c([\omega]),i_c([\omega]))=(\lL([\vartheta]),\lL_c([\vartheta]),i_c([\vartheta]))$.
		If $[c]=0$, then $i_c([\omega])=i_c([\vartheta])=0$ and $\lL_c([\omega])=\lL([\omega])=\lL([\vartheta])=\lL_c([\vartheta])$; we conclude by applying Proposition \ref{prop:lL}.  Assume now that $[c]\neq 0$.  We note that $\lL([\omega])=(n+1,0)$ if and only if $M_{\omega}=G$, equivalently $[\omega]=0$. In particular, if $\lL([\omega])=\lL([\vartheta])=(n+1,0)$, then $[\omega]=[\vartheta]$. Assume now that $\lL([\omega])=\lL([\vartheta])\neq(n+1,0)$ and so $[\omega]$ and $[\vartheta]$ are non-trivial. We finish by applying Propositions \ref{prop:ic=0} and \ref{prop:ic=1}.
	\end{proof}

	\section{Main result and applications}\label{sec:apps}

	We devote the present section to the proof of our main Theorem \ref{th:main} and to presenting some of its applications. In Sections \ref{subsec:2gen} and \ref{subsec:3gen} we explicitly compute the orbit sizes of the action of $A$ on $\hc^2(G;\F_p)$ respectively in the cases of $2$-generated and $3$-generated abelian $p$-groups, equivalently the cases when $r=1$ resp.\ $r=2$. We remark that in such cases the sizes of orbits are polynomial in $p$. We do not discuss the case of cyclic $G$, i.e. $r=0$, as in such case $\hc^2(G;\F_p)=\hab(G;\F_p)$; see Section \ref{subsec:strongisoLongexact}.  In Section \ref{subsec:moregens}, we collect some general remarks regarding the computability of the $A$-orbits in $\hc^2(G;\F_p)$.
	Until the end of Section \ref{sec:apps}, we denote by $\cor{O}$ the collection of orbits of the action of $A$ on $\hc^2(G;\F_p)$ and by $\mathfrak{S}=(|o|)_{o\in\cor{O}}$ the vector of the orbit sizes. For a more informative presentation of the data, the vector $\mathfrak{S}$ will be decorated by vertical bars to isolate
	\begin{itemize}
		\item the vector $\mathbf{o}$ of orbits associated to elements of $\hab(G;\F_p)$,
		\item each vector of orbits derived from a fixed orbit choice in $\hab(G;\F_p)$, following the order in $\mathbf{o}$.
	\end{itemize}
	Redundant brackets are ignored in the display of $\mathfrak{S}$.
	
	\subsection{The main theorem}\label{subsec:main}
	
	The following is our main result, which gives a combinatorial description of the $A$-orbits of the $A$-stable subset $\hab(G;\F_p)\times\im\cup$ of $\hc^2(G;\F_p)$. 
	
	\begin{theorem}\label{th:main}
		Let $[c],[d]\in\hab(G;\F_p)$ and $[\omega], [\vartheta]\in\im\cup$. Then the following are equivalent:
		\begin{enumerate}[label=$(\arabic*)$]
			\item $[c]+[\omega] \sim_A [d]+[\vartheta]$, and
			\item $(\lL([c]),\lL([\omega]), \lL_c([\omega]),  i_c([\omega]))=(\lL([d]),\lL([\vartheta]), \lL_d([\vartheta]),  i_d([\vartheta]))$.
		\end{enumerate} 
	\end{theorem}
	
	\begin{proof}
		$(1)\Rightarrow(2)$ Assume that $[c]+[\omega]\sim_A [d]+[\vartheta]$ and let $a=(\sigma,\lambda)\in A$ be such that 
		$a\cdot [c]+a\cdot [\omega]=a\cdot ([c]+[\omega])=[d]+[\vartheta]$. 
		With the notation from Proposition \ref{prop:LevelsTindices}, we then have that 
		$[\vartheta]=a\cdot[\omega]=[\omega_a]$
		and thus $\lL([\omega])=\lL([\vartheta])$, $\lL_c([\omega])=\lL_d([\vartheta])$, $\lL([c])=\lL([d])$, and $i_c([\omega])=i_d([\vartheta])$. 
		
		$(2)\Rightarrow(1)$ Assume that $\lL([\omega])=\lL([\vartheta])$, $\lL_c([\omega])=\lL_d([\vartheta])$, $\lL([c])=\lL([d])$, and $i_c([\omega])=i_d([\vartheta])$. Then, thanks to Proposition \ref{prop:abelian}, there exists $a\in A$ such that $a\cdot[c]=[d]$. Fix such $a$. Then, by Proposition \ref{prop:LevelsTindices}, we have that 
		$a\cdot ([c]+[\omega])=[d]+[\omega_a]$
		and, as a consequence, also that $\lL([\vartheta])=\lL([\omega_a])$, $\lL_d([\vartheta])=\lL_d([\omega_a])$, and $i_d([\vartheta])=i_d([\omega_a])$. Now, 
		Proposition \ref{prop:imcup} yields that there exists $a'\in A_d$ such that $a'\cdot[\omega_a]=[\vartheta]$ and thus such that 
		$a'a\cdot ([c]+[\omega])=[d]+[\vartheta]$. 
	\end{proof}
	
	\noindent
	We remark that, in view of Proposition \ref{prop:abelian}, one could replace $\lL([c])$ in Theorem \ref{th:main} with any of $\ell([c])$ or $\L([c])$ and, symmetrically, $\lL([d])$ with $\ell([d])$ or $\L([d])$. 
	We explicitly compute the vectors in Theorem \ref{th:main}(2) in Sections \ref{subsec:2gen} and \ref{subsec:3gen}, in the case when $G$ has a minimal generating set of $2$ or $3$ elements, respectively. 
	It would be interesting to understand the combinatorial nature of the collection of such vectors for an arbitrary number of generators.
	
	\subsection{The case of $2$-generated groups}\label{subsec:2gen}
	
	Assume that $G=\Z/(p^{m_1})\oplus\Z/(p^{m_2})$ for positive integers $m_1\leq m_2$ and, in the case that $p=2$, assume that $m_1>1$.
	We will show that the following hold:
	\[
	|\cor{O}|=
	\begin{cases}
		4 & \textup{if }  m_1=m_2, \\
		6 & \textup{otherwise },
	\end{cases}
	\]
	and
	\[
	\mathfrak{S}=
	\begin{cases}
		(1,p^2-1\mid p-1,(p-1)(p^2-1)) & \textup{if } m_1=m_2, \\
		(1,p-1,p^2-p\mid p-1,(p-1)^2,(p-1)(p^2-p)) & \textup{otherwise}.
	\end{cases}
	\]
	Thanks to Proposition \ref{prop:abelian}, the subspace $\hab(G;\F_p)$ consists of $2$ or $3$ orbits under $A$ respectively when $m_1=m_2$ or $m_1\neq m_2$. Let now $[\omega]\in\im\cup$. Then we have that 
	\[
	M_\omega=\begin{cases}
		G &\textup{if } [\omega]=0, \\
		pG &\textup{ otherwise},
	\end{cases}
	\]
	and, in particular, $i_c([\omega])=1$ if and only if $[c]\neq 0$ and $[\omega]=0$.
	Since both $G$ and $pG$ are characteristic in $G$, it follows from Lemma \ref{lem:lambda=1} that, for each $[c]\in\hab(G;\F_p)$, the set $\im\cup$ is the union of two orbits under $A_c$ with cardinalities $1$ and $p-1$. Now, the cup product being surjective (see Section \ref{subsec:plucker}) onto $\gen{\im\cup}$, it follows that the number of orbits is twice the number of orbits in $\hab(G;\F_p)$ and their sizes are 
	\[
	\mathfrak{S}=
	\begin{cases}
		(1,p^2-1\mid p-1, (p-1)(p^2-1)) & \textup{if } m_1=m_2, \\
		(1,p-1,p^2-p\mid p-1,(p-1)^2,(p-1)(p^2-p)) & \textup{otherwise}.
	\end{cases}
	\]
	For completeness, we include the levels-indices vectors from Theorem \ref{th:main}(2).
	If  \fbox{$m_1=m_2$}, then we have
	\begin{center}
		\begin{tabular}{c|c|c}
			& $[\omega]=0$ & $[\omega]\neq 0$ \\
			\hline
			$[c]=0$ & $(m_1+1,0\mid  m_1+1,0\mid m_1+1,0\mid 0)$ & $(m_1+1,0\mid m_1,m_1\mid m_1,m_1\mid  0)$ \\
			\hline
			$[c]\neq 0$ & $(m_1,m_1\mid  m_1+1,0\mid m_1+1,0\mid 1)$ & $(m_1,m_1\mid m_1,m_1\mid m_1,m_1\mid 0)$
		\end{tabular}
	\end{center}
	while, if \fbox{$m_1\neq m_2$}, the vectors are
	\begin{center}
		\begin{tabular}{c | c| c}
			& $[\omega]=0$ & $[\omega]\neq 0$ \\
			\hline
			$[c]=0$ & $(m_2+1,0\mid m_2+1,0\mid m_2+1,0\mid 0)$ & $(m_2+1,0\mid m_2,m_2\mid m_2,m_2\mid  0)$ \\
			\hline
			$[c]=\beta(\gamma_1^*)$ & $(m_1,m_1\mid  m_2+1,0\mid m_2+1,0\mid 1)$ & $(m_1,m_1\mid m_1,m_2\mid m_2,m_2\mid 0)$ \\
			\hline
			$[c]=\beta(\gamma_2^*)$ & $(m_2,m_2\mid  m_2+1,0\mid m_2,0\mid 1)$ & $(m_2,m_2\mid m_1,m_2\mid m_1,m_1\mid 0)$
		\end{tabular}
	\end{center}

	\subsection{The case of $3$-generated groups}\label{subsec:3gen}
	
	Assume that $G=\Z/(p^{m_1})\oplus\Z/(p^{m_2})\oplus\Z/(p^{m_3})$ where $m_1\leq m_2\leq m_3$ are positive integers with the additional condition that, if $p=2$, then $m_1>1$. We will show that the following hold: 
	\[
	\mid \cor{O}\mid =
	\begin{cases}
		5 & \textup{if }  m_1=m_2=m_3, \\
		11 & \textup{if } m_1<m_2=m_3, \\
		11 & \textup{if } m_1=m_2<m_3, \\
		19 & \textup{if } m_1<m_2<m_3.
	\end{cases}
	\]
	We will, additionally, give  the orbit sizes in each of the listed cases. For this, note that, as a consequence of Proposition \ref{prop:abelian}, the sizes of the $A$-orbits of $\hab(G;\F_p)$ are
	\[
	\mathfrak{S}_{\mathrm{ab}}=
	\begin{cases}
		(1,p^3-1) & \textup{if }  m_1=m_2=m_3, \\
		(1,p-1,p^3-p) & \textup{if } m_1<m_2=m_3, \\
		(1,p^2-1,p^3-p^2) & \textup{if } m_1=m_2<m_3, \\
		(1,p-1,p^2-p,p^3-p^2) & \textup{if } m_1<m_2<m_3.
	\end{cases}
	\]
	We proceed by looking at the specific cases, one by one. For this, observe that $\im\cup =\gen{\im\cup}$ and $\dim_{\F_p}\im\cup=3$; see Sections \ref{subsec:strongisoLongexact} and \ref{subsec:plucker}.
	\vspace{5pt}\\
	\noindent
	We start by assuming that \fbox{$m_1=m_2=m_3$}. 
	Let $[c]\in\graffe{0,\beta(\gamma_1^*)}$ and write $[\omega]$ for a generic element in $\im\cup$. 
	Then, following the notation in Theorem \ref{th:main}(2), we obtain the following possible values parametrizing the $A$-orbits in $\hc^2(G;\F_p)$:
	\begin{center}
		\begin{tabular}{c| c| c}
			& $[0]$ & $[\omega]\neq 0$ \\
			\hline
			$[c]=0$ & $(m_1+1,0\mid m_1+1,0\mid m_1+1,0\mid 0)$ & $(m_1+1,0\mid m_1,m_1\mid m_1+1,0\mid 0)$ \\
			\hline
			$[c]\neq 0$ & $(m_1,m_1\mid m_1+1,0\mid m_1+1,0\mid 1)$ & $(m_1,m_1\mid m_1,m_1\mid m_1,m_1 \mid 0)$\\
			&& $(m_1,m_1\mid m_1,m_1\mid m_1,m_1 \mid 1)$
		\end{tabular}
	\end{center}
	In particular, $\im\cup\setminus\graffe{0}$ consists of a unique $A$-orbit of cardinality $p^3-1$. Assume now that $[c]=\beta(\gamma_1^*)$. In this case, we obtain
	\begin{itemize}
		\item $\cor{I}_0=\graffe{ [\omega]\in \im\cup\setminus \graffe{0} : i_c([\omega])=0}=\graffe{\lambda_1 [v_{11}^*\cup v_{21}^*]+\lambda_2[v_{11}^*\cup v_{31}^*] : \lambda_i\in\F_p, (\lambda_1,\lambda_2)\neq(0,0)},$
		\item $\cor{I}_1=\graffe{ [\omega]\in \im\cup\setminus \graffe{0} : i_c([\omega])=1}=\graffe{\lambda_1[v_{11}^*\cup v_{21}^*]+\lambda_2[v_{11}^*\cup v_{31}^*]+\lambda_3[v_{21}^*\cup v_{31}^*] : \lambda_i\in\F_p,\, \lambda_3\neq 0}.$
	\end{itemize}
	It follows that $|\cor{I}_0|=p^2-1$ and $|\cor{I}_1|=p^3-p^2$ and thus Proposition \ref{prop:imcup} yields that 
	\[
	\mathfrak{S}=(1,p^3-1 \mid p^3-1\mid (p^3-1)(p^2-1),(p^3-1)(p^3-p^2)).
	\]
	Assume now that \fbox{$m_1<m_2=m_3$}. Define $[c_1]=\beta(\gamma_1^*)$ and $[c_2]=\beta(\gamma_2^*)$. Write, moreover, $[\omega]$ for a generic element in $\im\cup$. 
	Then, following the notation in Theorem \ref{th:main}(2), the values parametrizing the $A$-orbits in $\hc^2(G;\F_p)$ are collected below:
	\begin{center}
		\begin{tabular}{c | c| c}
			& $[\omega]=0$ & $[\omega]\neq 0$ \\
			\hline
			$[c]=0$ & $(m_2+1,0 \mid m_2+1,0 \mid m_2+1,0 \mid 0)$ & $(m_2+1,0\mid m_1,m_2 \mid m_1,m_2 \mid  0)$ \\
			&& $(m_2+1,0\mid m_2,m_2 \mid m_2,m_2\mid  0)$ \\
			\hline
			$[c_1]=\beta(\gamma_1^*)$ & $(m_1,m_1\mid m_2+1,0 \mid m_2+1,0\mid 1)$ & $(m_1,m_1\mid m_1,m_2 \mid m_2,m_2\mid 0)$ \\
			&& $(m_1,m_1\mid m_1,m_2 \mid m_2,m_2\mid 1)$ \\
			&& $(m_1,m_1\mid m_2,m_2 \mid m_2,m_2\mid 1 )$ \\
			\hline
			$[c_2]=\beta(\gamma_2^*)$ & $(m_2,m_2\mid m_2+1,0 \mid m_2+1,0 \mid 1 )$ & $(m_2,m_2 \mid m_1,m_2 \mid m_1,m_1 \mid 0)$ \\
			&& $(m_2,m_2 \mid m_1,m_2 \mid m_1,m_2\mid 1)$ \\
			&& $(m_2,m_2 \mid m_2,m_2 \mid  m_2,m_2\mid 0)$
		\end{tabular}
	\end{center}
	Note also that the two $A$-orbits in $\im \cup\setminus\{0\}$ are represented by $[v_{11}^*\cup v_{21}^*]$ and $[v_{21}^*\cup v_{22}^*]$ and correspond respectively to the $G$-levels $(m_1,m_2)$ and $(m_2,m_2)$.
	It is a straightforward computation to show that the following hold:
	\begin{align*}
		\cor{I}^1_0(m_1,m_2)&=\graffe{ [\omega]\in \im\cup\setminus \graffe{0} : \lL([\omega])=(m_1,m_2),  i_{c_1}([\omega])=0} \\
		&=\graffe{\lambda_1 [v_{11}^*\cup v_{21}^*]+\lambda_2[v_{11}^*\cup v_{22}^*] : \lambda_i\in\F_p, (\lambda_1,\lambda_2)\neq(0,0)}, \\
		\cor{I}^1_1(m_1,m_2)&=\graffe{ [\omega]\in \im\cup\setminus \graffe{0} :\lL([\omega])=(m_1,m_2), i_{c_1}([\omega])=1} \\
		&=\graffe{\lambda_1[v_{11}^*\cup v_{21}^*]+\lambda_2[v_{11}^*\cup v_{22}^*]+\lambda_3[v_{21}^*\cup v_{22}^*] : \lambda_i\in\F_p,\, (\lambda_1,\lambda_2)\neq(0,0), \lambda_3\neq 0}, \\
		\cor{I}^1_1(m_2,m_2)& =\graffe{ [\omega]\in \im\cup\setminus \graffe{0} : \lL([\omega])=(m_2,m_2), i_{c_1}([\omega])=1}\\
		&=\graffe{\lambda_3[v_{21}^*\cup v_{22}^*] : \lambda_3\in\F_p,\, \lambda_3\neq 0}, \\
		\cor{I}^2_0(m_1,m_2) & =\graffe{ [\omega]\in \im\cup\setminus \graffe{0} : \lL([\omega])=(m_1,m_2), i_{c_2}([\omega])=0} \\
		& =\graffe{\lambda_1 [v_{11}^*\cup v_{21}^*]+\lambda_2[v_{21}^*\cup v_{22}^*] : \lambda_i\in\F_p, \lambda_1\neq 0}, \\
		\cor{I}^2_0(m_2,m_2) & =\graffe{ [\omega]\in \im\cup\setminus \graffe{0} : \lL([\omega])=(m_2,m_2), i_{c_2}([\omega])=0}\\
		& =\graffe{\lambda_2[v_{21}^*\cup v_{22}^*] : \lambda_2\in\F_p, \lambda_2\neq 0}, \\
		\cor{I}^2_1(m_1,m_2) &=\graffe{ [\omega]\in \im\cup\setminus \graffe{0} :\lL([\omega])=(m_1,m_2), i_{c_2}([\omega])=1}\\
		&=\graffe{\lambda_1[v_{11}^*\cup v_{21}^*]+\lambda_2[v_{21}^*\cup v_{22}^*]+\lambda_3[v_{11}^*\cup v_{22}^*] : \lambda_i\in\F_p,\, \lambda_3\neq 0}.
	\end{align*}
	It follows that 
	\begin{alignat*}{3}
		|\cor{I}^1_0(m_1,m_2)| &=p^2-1,\quad |\cor{I}^1_1(m_2,m_2)| &= p-1, \quad   |\cor{I}^1_1(m_1,m_2)|& = p^3-p^2-p+1,  \\
		|\cor{I}^2_0(m_1,m_2)|& = p^2-p, \quad   |\cor{I}^2_0(m_2,m_2)|& = p-1, \quad  |\cor{I}^2_1(m_1,m_2)|& =p^3-p^2,
	\end{alignat*}
	and so we derive from our table of possibilities 
	and Theorem \ref{th:main} that 
	\begin{align*}
		\mathfrak{S} =(&1, p-1, p^3-p\mid p^3-p, p-1 \mid  \\
		& (p-1)(p^2-1), (p-1)^2,  (p-1)(p^3-p^2-p+1) \mid  \\
		& (p^3-p)(p^2-p), (p^3-p)(p-1), (p^3-p)(p^3-p^2)).
	\end{align*}
	We have developed the current case in full detail to show how Theorem \ref{th:main} yields the orbit count. One can compute the orbit sizes in the remaining cases in a similar manner and so we present them in a slightly more synthetic way.
	\vspace{5pt}\\
	\noindent
	Assume that \fbox{$m_1=m_2<m_3$}. 
	Write $[c_1]=\beta(\gamma_1^*)$ and $[c_2]=\beta(\gamma_2^*)$. We also write $[\omega]$ for a generic element in $\im\cup$. 
	Then, following the notation in Theorem \ref{th:main}(2), the values parametrizing the $A$-orbits in $\hc^2(G;\F_p)$ are listed in the next table:
	\begin{center}
		\begin{tabular}{c | c| c}
			& $[\omega]=0$ & $[\omega]\neq 0$ \\
			\hline
			$[c]=0$ & $( m_3+1,0\mid m_3+1,0  \mid m_3+1  \mid 0)$ & $(m_3+1,0\mid m_1,m_1  \mid m_1,m_1  \mid  0)$ \\
			&& $( m_3+1,0\mid m_1,m_3 \mid m_1,m_3 \mid  0)$ \\
			\hline
			$[c_1]=\beta(\gamma_1^*)$ & $(m_1,m_1\mid m_3+1,0 \mid m_3+1, 0 \mid 1)$ & $(m_1,m_1\mid m_1,m_1  \mid m_1,m_1 \mid 0 )$ \\
			&& $(m_1,m_1\mid m_1,m_1  \mid m_1,m_3  \mid 1 )$\\
			&& $(m_1,m_1\mid m_1,m_3  \mid m_3,m_3 \mid 0)$ \\
			&& $(m_1,m_1\mid m_1,m_3 \mid m_1,m_3 \mid 1 )$ \\
			\hline
			$[c_2]=\beta(\gamma_2^*)$ & $(m_3,m_3\mid m_3+1,0 \mid m_3,0 \mid 1 )$ & $(m_3,m_3 \mid m_1,m_1  \mid m_1,m_1  \mid 1)$ \\
			&& $(m_3,m_3 \mid m_1,m_3 \mid m_1,m_1 \mid 0)$ 
		\end{tabular}
	\end{center}
	We observe that the two $A$-orbits in $\im \cup\setminus\{0\}$ are represented by $[v_{11}^*\cup v_{12}^*]$ and $[v_{11}^*\cup v_{21}^*]$ and correspond respectively to the $G$-levels $(m_1,m_1)$ and $(m_1,m_3)$: these orbits have sizes respectively 
	$p^2-1$ and $p^3-p^2$. Analogously to the previous case, one can compute that
	\begin{align*}
		\mathfrak{S} =(&1, p^2-1, p^3-p^2\mid p^3-p^2,p^2-1 \mid  \\
		& (p^2-1)(p^2-p),(p^2-1)(p^3-2p^2+p), (p^2-1)(p-1),(p^2-1)(p^2-p) \mid  \\
		&(p^3-p^2)^2,(p^3-p^2)(p^2-1) ).
	\end{align*}
	We conclude with the case \fbox{$m_1<m_2<m_3$}. Write $[c_1]=\beta(\gamma_1^*)$,  $[c_2]=\beta(\gamma_2^*)$, and $[c_3]=\beta(\gamma_3^*)$. 
	Analogously to the previous cases, we collect the possible levels-indices vectors from Theorem \ref{th:main}(2) in the next table:
	\begin{center}
		\begin{tabular}{c | c| c}
			& $[\omega]=0$ & $[\omega]\neq 0$ \\
			\hline
			$[c]=0$ & $( m_3+1,0\mid m_3+1,0  \mid m_3+1  \mid 0)$ & $(m_3+1,0\mid m_1,m_2  \mid m_1,m_2  \mid  0)$ \\
			&& $( m_3+1,0\mid m_1,m_3 \mid m_1,m_3 \mid  0)$ \\
			&& $( m_3+1,0\mid m_2,m_3 \mid m_2,m_3 \mid  0)$ \\
			\hline
			$[c_1]=\beta(\gamma_1^*)$ & $(m_1,m_1\mid m_3+1,0\mid m_3+1,0 \mid 1)$ & $(m_1,m_1\mid m_1,m_2  \mid m_2,m_2 \mid  0)$ \\
			&& $(m_1,m_1\mid m_1,m_2  \mid m_2,m_3  \mid 1)$ \\
			&& $(m_1,m_1\mid m_1,m_3  \mid m_3,m_3 \mid 0)$ \\
			&& $(m_1,m_1\mid m_1,m_3 \mid m_2,m_3 \mid  1)$ \\
			&& $(m_1,m_1\mid m_2,m_3 \mid m_2,m_3 \mid  1)$ \\
			\hline
			$[c_2]=\beta(\gamma_2^*)$ & $(m_2,m_2\mid m_3+1,0 \mid  m_3+1,0\mid 1 )$ & $( m_2,m_2\mid m_1,m_2  \mid m_1,m_1  \mid 0)$ \\
			&& $( m_2,m_2\mid m_1,m_2  \mid m_1,m_3  \mid 1)$ \\
			&& $( m_2,m_2\mid m_1,m_3  \mid m_1,m_3  \mid 1)$ \\
			&& $( m_2,m_2\mid m_2,m_3  \mid m_3,m_3                                                                                                                                                                                                                                                                                                      \mid 0)$ \\
			\hline
			$[c_3]=\beta(\gamma_3^*)$ & $(m_3,m_3\mid m_3+1,0 \mid m_3,0 \mid 1)$ & $(m_3,m_3\mid m_1,m_2  \mid m_1,m_2 \mid 1 )$ \\
			&& $(m_3,m_3\mid  m_1,m_3 \mid m_1,m_1 \mid  0)$ \\
			&& $(m_3,m_3\mid  m_2,m_3 \mid m_2,m_2 \mid 0 )$ 
		\end{tabular}
	\end{center}
	Representatives of the $A$-orbits of $\im\cup\setminus\graffe{0}$ are $[v_{11}^*\cup v_{21}^*]$, $[v_{11}^*\cup v_{31}^*]$, and $[v_{21}^*\cup v_{31}^*]$ corresponding respectively to the levels $(m_1,m_2)$, $(m_1,m_3)$, and $(m_2,m_3)$.
	As a consequence, one computes that
	\begin{align*}
		\mathfrak{S} =(&1,p-1, p^2-p, p^3-p^2 \mid p^3-p^2, p^2-p,p-1  \mid  \\
		& (p-1)^2p,(p-1)^3p,(p-1)^2,(p-1)^3,(p-1)^2 \mid  \\
		& (p^2-p)^2,(p^2-p)(p^3-2p^2+p),(p^2-p)^2,(p^2-p)(p-1) \mid \\
		& (p^3-p^2)^2,(p^3-p^2)(p^2-p),(p^3-p^2)(p-1)).
	\end{align*}

	\subsection{Higher number of generators}\label{subsec:moregens}
	
	In Sections \ref{subsec:2gen} and \ref{subsec:3gen}, we have made use of Theorem \ref{th:main} to compute the orbit sizes of the action of $A$ on $\hc^2(G;\F_p)$. As the careful reader might have observed, however, we did not need the full information from the vectors in Theorem \ref{th:main}(2) to exploit the cases of $2$- and $3$-generated groups. In the case of $2$-generated groups, the $c$-levels and $c$-index can always be derived from the knowledge of $\lL([c])$ and $\lL([\omega])$ because $V$ has dimension $2$. In the case of $3$-generated groups, the knowledge of the vector $(\lL([c]),\lL([\omega]),i_c([\omega]))$ suffices for the computation of $\lL_c([\omega])$ because $V$ has only dimension $3$. 
	When $G$ requires a generating set of larger cardinality, the full information carried by the vectors described in Theorem \ref{th:main}(2) is needed, as the following example shows. 
	
	\begin{ex}
		Assume $p$ is odd and $G$ is given by
		\[
		G= \Z/(p)\oplus (\Z/(p^2))^2\oplus \Z/(p^3)\oplus(\Z/(p^4))^2=\langle \gamma_{11}, \gamma_{21},\gamma_{22},\gamma_{31},\gamma_{41},\gamma_{42}\rangle.
		\]
		Let moreover $T,M,M'$ be subgroups of $G$ given by
		\begin{align*}
			T&=\langle \gamma_{11},\gamma_{22},\gamma_{31},\gamma_{41}, \gamma_{42} \rangle+pG, \\
			M&=\langle \gamma_{11},\gamma_{21}-\gamma_{31},\gamma_{22},\gamma_{42} \rangle +pG,\\
			M'&=\langle \gamma_{11}, \gamma_{21}, \gamma_{31}, \gamma_{41} \rangle+pG,
		\end{align*}
		and observe that $T$ is maximal in $G$, while $G/M$ and $G/M'$ are elementary abelian of rank $2$. Additionally, we have that $\lL(M)=(2,4)=\lL(M')$ and 
		\begin{align*}
			M\cap T&=\langle \gamma_{11},\gamma_{22},\gamma_{42} \rangle+pG  \textup{ and }  M'\cap T=\langle \gamma_{11},\gamma_{31},\gamma_{41} \rangle+pG
		\end{align*} 
		and so, in particular, $M$ and $M'$ are not contained in $T$. Equivalently, if $[c]\in\hab(G;\F_p)$ represents $T$ via \eqref{eq:TauMap}, then $i_c(M)=i_c(M')=1$.
		Nevertheless, the $T$-levels of $M$ and $M'$ do not coincide: indeed one can compute
		$\lL_T(M)=(3,4)\neq(2,4)=\lL_T(M').$
	\end{ex}
	
	\noindent
	We close the current section and the paper with some observations concerning the determination of the $A$-orbits in $\hc^2(G;\F_p)$ for arbitrary $G$. Our main theorem allows us to compute the orbits contained in $\hab(G;\F_p)\times\im\cup$, which is -- for $\dG(G)\geq 4$ -- a proper subset of $\hc^2(G;\F_p)$. A key ingredient in the proof of Theorem \ref{th:main} is 
	Corollary \ref{cor:stabT} and we are confident that a generalization of it to elements of higher rank in $\mathbb{P}\gen{\im\cup}$ will yield a description of the orbits of $\hc^2(G;\F_p)$. 
	Such a generalization will most likely build upon the
	geometry of $\mathbb{P}(\Lambda^2V)$, which is very well-understood, via identifying its elements with equivalence classes of tuples of subspaces of $V$. We hope to come back to this interesting problem in a future paper.

\providecommand{\bysame}{\leavevmode\hbox to3em{\hrulefill}\thinspace}
\providecommand{\MR}{\relax\ifhmode\unskip\space\fi MR }
\providecommand{\MRhref}[2]{%
  \href{http://www.ams.org/mathscinet-getitem?mr=#1}{#2}
}
\providecommand{\href}[2]{#2}

	\vspace*{2em}
	\noindent
	{\footnotesize
		\begin{minipage}[t]{0.43\textwidth}
			Oihana Garaialde Oca\~na \\
			Matematika Saila, 
			Euskal Herriko Unibertsitatearen Zientzia eta Teknologia Fakultatea \\
			Posta-kutxa 644 \\
			48080 Bilbo \\
			Spain \\
			\quad\\
			E-mail:  \href{mailto:oihana.garayalde@ehu.eus}{oihana.garayalde@ehu.eus}
		\end{minipage}
		\hfill
		\begin{minipage}[t]{0.36\textwidth}
			Mima Stanojkovski\\
			Universit\`a di Trento\\
			Dipartimento di Matematica\\
			via Sommarive 14\\
			38123 Povo di Trento\\
			Italia\\
			\quad\\
			E-mail: \href{mailto:mima.stanojkovski@unitn.it}{mima.stanojkovski@unitn.it}
		\end{minipage}
	}

\end{document}